\documentclass[11pt]{article}
\usepackage{setspace}
\usepackage{amsmath,amssymb}
\usepackage{amsthm}
\usepackage{algorithm, algpseudocode}
\usepackage{url}
\usepackage{fullpage}
\usepackage[vmargin=2.5cm]{geometry}

\newtheorem{claim}{Claim}[section]
\newtheorem{lemma}[claim]{Lemma}

\newtheorem{theorem}{Theorem}
\newtheorem{proposition}[claim]{Proposition}
\newtheorem{corollary}[claim]{Corollary}
\newtheorem{definition}[claim]{Definition}
\newtheorem{remark}[claim]{Remark}

\numberwithin{equation}{section}

\def\<{\langle}
\def\>{\rangle}

\def\G{{\mathcal G}}

\def\root{{\phi}}

\def\eps{{\varepsilon}}

\def\sT{{\sf T}}

\def\P{{\mathbb P}}
\def\prob{{\mathbb P}}

\def\naturals{{\mathbb N}}
\def\E{{\mathbb E}} 
\def\Var{{\sf{Var}}}

\def\reals{\mathbb{R}}

\def\normal{{\sf N}}

\def\Var{{\sf Var}}

\def\cT{{\cal T}}

\def\de{{\rm d}}

\def\cR{\mathcal{R}}

\def\Ball{{\sf Ball}}

\def\Var{{\sf Var}}

\def\ind{\mathbb{I}}
\def\bq{\mathbf{q}}
\newcommand\norm[1]{\lVert{#1}\rVert}
\def\bs{\backslash}

\newcommand\myeqref[1]{{Eq.\,\eqref{#1}}}

\def\F{{\mathcal F}}

\def\sC{{\sf C}}
\def\hsC{{\widehat{\sf C}}}

\def\cF{{\cal F}}
\def\mfF{{\mathfrak{F}}}

\def\bG{{\bf G}}
\def\Gs{{\bG\rvert_{\sC_N}}}
\def\Gsc{{\bG\rvert_{\sC^c_N}}}
\def\cU{{\mathcal{U}}}
\def\ed{\stackrel{{\rm d}}{=}}
\def\tA{\widetilde{A}}
\def\tkappa{\widetilde{\kappa}}
\def\hkappa{\widehat{\kappa}}
\def\tgamma{\widetilde{\gamma}}

\def\Var{{\rm Var}}
\def\md{{\mathrm{d}}}
\def\convD{{\,\stackrel{\mathrm{d}}{\Rightarrow} \,}}
\def\oxi{\overline{\xi}}
\def\Tree{{\sf Tree}}
\def\di{{\partial i}}
\def\tsC{\widetilde{\sf C}}
\def\brho{\overline{\rho}}
\def\sB{{\sf B}}
\def\hmu{\hat{\mu}}

\title{Finding Hidden Cliques of Size $\sqrt{N/e}$ in Nearly Linear Time} 
\author{Yash~Deshpande\thanks{Y.~Deshpande is with the Department of Electrical 
Engineering, Stanford University}
~and~Andrea~Montanari\thanks{A.~Montanari is with the Departments of Electrical 
Engineering and Statistics, Stanford University}}

\begin{document}
\maketitle
\begin{abstract}
Consider an Erd\"os-Renyi random graph in which each edge is present
independently with probability $1/2$, except for a subset $\sC_N$ of the
vertices that form a clique (a completely connected subgraph). We
consider the problem of identifying the clique, given a realization of
such a random graph. 

The best known algorithm provably finds the clique in linear time with high
probability, provided $|\sC_N|\ge 1.261\sqrt{N}$ \cite{dekel2011finding}. 
Spectral methods can be shown to fail on cliques smaller than $\sqrt{N}$. In this paper we
describe a nearly linear time algorithm that succeeds with high probability for
$|\sC_N|\ge (1+\eps)\sqrt{N/e}$ for any $\eps>0$. This is the first
algorithm that provably improves over spectral methods.

We further generalize the hidden clique problem to other
background graphs (the standard case corresponding to the complete graph on
$N$ vertices). For large girth regular graphs of degree $(\Delta+1)$ we
prove that `local' algorithms succeed if $|\sC_N|\ge
(1+\eps)N/\sqrt{e\Delta}$ and
fail if $|\sC_N|\le(1-\eps)N/\sqrt{e\Delta}$. 
\end{abstract}

\section{Introduction}

Numerous modern data sets have network structure, i.e. the dataset
consists of observations on
pairwise relationships among a set of $N$ objects. A recurring 
computational problem in this context is the one of identifying a small subset
of `atypical' observations against a noisy background. This paper develops
a new type of algorithm and analysis for this problem. In particular
we improve over the best methods for finding a hidden clique in an
otherwise random graph.

Let $G_N=([N], E_N)$ be a graph over the vertex set $[N]\equiv
\{1,2,\dots,N\}$ and $Q_0$, $Q_1$ be two distinct probability distributions over
the real line $\reals$. Finally, let $\sC_N\subseteq [N]$ be a subset
of vertices uniformly random given its size $|\sC_N|$. For each edge
$(i,j)\in E_N$ we draw an independent random variable $W_{ij}$ with
distribution $W_{ij}\sim Q_1$ if both $i\in \sC_N$ and $j\in\sC_N$ and
$W_{ij}\sim Q_0$ otherwise. The \emph{hidden set problem}
is to identify the set $\sC_N$ given knowledge of the graph
$G_N$ and the observations $W=(W_{ij})_{(ij)\in E_N}$. We will refer
to $G_N$ as to the  \emph{background graph}. We emphasize that
$G_N$ is non-random and that it carries no information about the hidden
set $\sC_N$.
 
In the rest of this introduction we will assume, for simplicity,
$Q_1=\delta_{+1}$ and $Q_0 = (1/2)\delta_{+1}+(1/2)\delta_{-1}$. In
other words, edges $(i,j)\in E_N$ with  endpoints $\{i,j\}\subseteq\sC_N$
are labeled with $W_{ij} = +1$. Other edges $(i,j)\in E_N$ have a
uniformly random label $W_{ij}\in\{+1,-1\}$.
Our general treatment in the next sections covers arbitrary subgaussian distributions  $Q_0$ and
$Q_1$ and does not require these distributions to be known in advance.

The special case $G_N=K_N$ (with $K_N$ the complete graph) has attracted considerable attention over
the last twenty years \cite{jerrum1992large} and is known as
as the \emph{hidden} or \emph{planted clique problem}. In this case, the
background graph does not play any role, and the random variables $W =
(W_{ij})_{i,j\in [N]}$ can be organized in an $N\times N$ symmetric
matrix (letting, by convention, $W_{ii} = 0$). The matrix $W$ can be
interpreted as the adjacency matrix of a random graph $\cR_N$ generated as
follows. Any pair of vertices $\{i,j\}\subseteq\sC_N$ is connected by an
edge. Any other pair $\{i,j\}\not\subseteq \sC_N$ is instead connected
independently with probability $1/2$.  (We use here $\{+1,-1\}$ instead of
$\{1,0\}$ for the entries of the adjacency matrix. This encoding  is unconventional
but turns out to be mathematically convenient.)
Due to the symmetry of the model, the set  $\sC_N\subseteq [N]$ does
not need to be random and can be chosen
arbitrarily in this case. 

It is easy to see that, allowing for exhaustive search, the hidden
clique can be found with high probability as soon as $|\sC_N|\ge
2(1+\eps)\log_2N$ for any $\eps>0$. This procedure has complexity
$\exp[\Theta((\log N)^2)]$. Viceversa, if  $|\sC_N|\le
2(1-\eps)\log_2N$, then the clique cannot be uniquely identified. 

Despite a large body of research, the best polynomial-time algorithms to date
require $|\sC_N|\ge c\sqrt{N}$ to succeed with high probability. This
was first achieved by Alon, Krivelevich and Sudakov
\cite{alon1998finding} through a spectral technique. It is useful to
briefly discuss this class of methods and their limitations. Letting 
$u_{\sC_N}\in\reals^N$ be the indicator vector on $\sC_N$ (i.e. the
vector with entries $(u_{\sC_N})_i = 1$ for $i\in\sC_N$ and $=0$
otherwise), we have 
\begin{align}
W = u_{\sC_N}u_{\sC_N}^{\sT} + Z - Z_{\sC_N,\sC_N}\, .\label{eq:RankDeformation}
\end{align}
Here $Z\in\reals^{N\times N}$ is a symmetric matrix with i.i.d. entries
$(Z_{ij})_{i<j}$ uniformly random in $\{+1,-1\}$ and $Z_{\sC_N,\sC_N}$ is
the matrix obtained by zeroing all the entries $Z_{ij}$ with
$\{i,j\}\not\subseteq\sC_N$.
Denoting by $\|A\|_2$ the $\ell_2$ operator norm of matrix $A$, we have 
$\|u_{\sC_N}u_{\sC_N}^{\sT}\|_2 = \|u_{\sC_N}\|_2^2 = |\sC_N|$. On the other
hand, a classical result by F\"uredi and Koml\"os
\cite{furedi1981eigenvalues} implies that, with high probability,
$\|Z\|_2 \le c_0\sqrt{N}$ and $\|Z_{\sC_N,\sC_N}\|_2 \le
c_0\sqrt{|\sC_N|}$. Hence, if $|\sC_N|\ge c\sqrt{N}$ with $c$ large
enough, the first term in the decomposition (\ref{eq:RankDeformation})
dominates the others. By a standard matrix perturbation argument \cite{davis1970sin},
letting $v_1$ denote the principal eigenvector of $W$, we have 
$\|v_1-u_{\sC_N}/\sqrt{|\sC_N|}\|_2\le \eps$ provided the constant $c=c(\eps)$ is chosen large
enough. It follows that selecting the $|\sC_N|$ largest entries of $v_1$
yields an estimate $\hsC_N\subset [N]$ that includes at least half of the
vertices of $\sC_N$: the other half can be subsequently identified
through a simple procedure \cite{alon1998finding}.

The spectral approach does not exploit the fact that $|\sC_N|$ is much
smaller than $N$ or --in other words-- the fact that $u_{\sC_N}$ is a
\emph{sparse} vector. Recent results in random matrix theory
suggest that it is unlikely that the same approach can be pushed to
work for $|\sC_N|\le (1-\eps)\sqrt{N}$ (for any $\eps>0$). For instance the following
is a consequence of \cite[Theorem
2.7]{knowles2011isotropic}.
(The proof is provided in Appendix \ref{sec:SpectralProof})

%
\begin{proposition}\label{propo:Spectral}
  Let $e_{\sC_N}= u_{\sC_N}/N^{1/4}$ be the normalized indicator vector on
the vertex set $\sC_N$, and $Z$ a Wigner random matrix with subgaussian
entries such that $E\{Z_{ij}\}=0$, $\E\{Z_{ij}^2\}=1/N$
Denote by $v_1,v_2,v_3,\dots,v_{\ell}$  the eigenvectors of
$W = u_{\sC_N}u_{\sC_N}^\sT + Z$,  corresponding to the $\ell$ 
largest eigenvalues.

Assume $|\sC_N|\ge (1+\eps)\sqrt{N}$ for some $\eps>0$. Then, with high probability,
$\<v_1,e_{\sC_N}\>\ge \min(\sqrt{\eps},\eps)/2$.
Viceversa, assume $|\sC_N|\le (1-\eps)\sqrt{N}$. Then, with high probability
for any fixed constant $\delta>0$, $|\<v_i,e_{\sC_N}\>|\le c\,N^{-1/2 + \delta}$ for all $i\in\{1,\dots,\ell\}$
and some $c=c(\eps,\ell)$.
\end{proposition}
  

In other words, for $|\sC_N|$ below $\sqrt{N}$ and any fixed $\ell$,  the first $\ell$ principal eigenvectors of $W$
are essentially no more correlated with the set $\sC_N$ than a random unit vector. 
A natural reaction to this limitation is to try to exploit the sparsity of
$u_{\sC_N}$. Ames and Vavasis \cite{ames2011nuclear} studied a convex
optimization formulation wherein $W$ is approximated by a sparse low-rank matrix.
These two objectives (sparsity and rank) are convexified through the usual $\ell_1$-norm
and nuclear-norm relaxations. These authors prove that this convex
relaxation approach is
successful with high probability, provided $|\sC_N|\ge c\sqrt{N}$ for an
unspecified constant $c$. A similar result follows from the robust PCA
analysis of Cand\'es, Li, Ma, and Wright \cite{candes2011robust}.

Dekel, Gurel-Gurevich and Peres \cite{dekel2011finding}
developed simple iterative schemes with $O(N^2)$ complexity (see also
\cite{feige2010finding} for similar approaches). For the best of their
algorithms, these authors prove that it succeeds with high probability
provided $|\sC_N|\ge 1.261\sqrt{N}$. Finally, \cite{alon1998finding}
also provide a simple procedure that, given an algorithm that is
successful for $|\sC_N|\ge c\sqrt{N}$ produces an algorithm that is
successful for $|\sC_N|\ge c\sqrt{N/2}$, albeit with complexity $\sqrt{N}$ times larger.

Our first result proves that the hidden clique can be
identified in nearly linear time well below the spectral threshold $\sqrt{N}$,
see Proposition \ref{propo:Spectral}.
\begin{theorem}\label{thm:MainClique}
Assume $|\sC_N|\ge (1+\eps)\sqrt{N/e}$, for some $\eps>0$ independent of $N$.
Then there exists a $O(N^2\log N)$ time algorithm that identifies the hidden
clique $\sC_N$ with high probability.
\end{theorem}
In Section \ref{sec:Complete} we will state and prove a generalization of this
theorem for arbitrary --not necessarily known-- distributions $Q_0$, $Q_1$.

Our algorithm is based on a quite different philosophy with respect to
previous approaches to the same problem. We aim at estimating optimally the set $\sC_N$ by computing the
posterior probability that $i\in\sC_N$, given edge data $W$. This is, in general,
$\#$P-hard and possibly infeasible if $Q_0$, $Q_1$ are
unknown. We therefore consider an algorithm derived from  \emph{belief propagation}, a
heuristic machine learning method for approximating posterior probabilities in
graphical models. We develop a rigorous analysis of this algorithm that is
asymptotically exact as $N\to\infty$, and prove that indeed the algorithm converges
to the correct set of vertices $\sC_N$ for $|\sC_N|\ge
(1+\eps)\sqrt{N/e}$. Viceversa, the algorithm converges to an
uninformative fixed point for $|\sC_N|\le (1-\eps)\sqrt{N/e}$.

\vspace{0.2cm}

Given Theorem \ref{thm:MainClique}, it is natural to ask whether the
threshold $\sqrt{N/e}$ has a fundamental computational meaning or is
instead only relevant for our specific algorithm. Recently,
\cite{feldman2012statistical} proved complexity lower bounds for the
hidden clique model, in a somewhat different framework. In the
formulation of \cite{feldman2012statistical}, one can query columns of
$W$ and a new realization from the distribution of $W$ given $\sC_N$ is
instantiated at each query. Assuming that each column is queried
$O(1)$ times, their lower bound would require $|\sC_N|\ge
N^{1/2-\eps}$. While this analysis can possibly be adapted to
our setting, it is unlikely to yield a lower bound of the form
$|\sC_N|\ge c\sqrt{N}$ with a sharp constant $c$.

Instead, we take a different point of view, and consider the hidden set problem on a
\emph{general background graph} $G_N$. Let us emphasize once more that
$G_N$ is non random and that all the information about the hidden set
is carried by the edge labels $W=(W_{lk})_{(l,k)\in E_N}$. 
In addition, we attach to the edges a collection of independent labels $U =
(U_{lk})_{(l,k)\in E_N}$ i.i.d. and uniform in $[0,1]$.
The $U$ labels exist to provide for (possible) randomization in the algorithm.
Given such
a graph $G_N$ with labels $W$, $U$,  a vertex $i\in [N]$ and
$t\ge 0$, we let $\Ball_{G_N}(i;t)$ denote the subgraph of $G_N$
induced by those vertices $j\in[N]$ whose graph distance from $i$ is
at most $t$. We regard $\Ball_{G_N}(i;t)$ as a graph rooted at $i$,
with edge labels $W_{jl}, U_{jl}$ inherited from $G_N$.
\begin{definition}
An algorithm for the hidden
set problem is said to be \emph{$t$-local} if, denoting by $\hsC_N$ its
output, the membership $(i\in \hsC_N)$ is a function of the neighborhood
$\Ball_{G_N}(i;t)$. We say that it is \emph{local} if it is $t$-local
for some $t$ independent of $N$.
\end{definition}
The concept of (randomized) local algorithms  was introduced in
\cite{angluin1980local} and formalizes the notion of an algorithm that
can be run in $O(1)$ time in a distributed network. We refer to 
\cite{linial1992locality,naor1995can} for earlier contributions, and to 
\cite{suomela2013survey} for a recent survey. 

We say that a sequence of graphs $\{G_N\}_{N\ge 1}$ is \emph{locally
  tree-like} if, for any $t\ge 0$, the fraction of vertices $i\in [N]$
such that $\Ball_{G_N}(i;t)$ is a tree converges to one as
$N\to\infty$. As a standard example, random regular graphs are 
locally tree-like. The next result is proved in Section \ref{sec:Sparse}.
\begin{theorem}\label{thm:Sparse}
Let $\{G_N\}_{N\ge 1}$ be a sequence of locally tree-like graphs, with
regular degree $(\Delta+1)$, and let $\sC_N\subseteq [N]$ be a uniformly random
subset of the vertices of given size $|\sC_N|$.
If $|\sC_N|\le (1-\eps)N/\sqrt{e\Delta}$ for some $\eps>0$,
there exists $\xi>0$ independent of $\Delta$ and $\eps$ such that 
any local algorithm outputs a set of vertices $\hsC_N$ with
$\E[|\sC_N\triangle \hsC_N|]\ge N\xi/\sqrt{\Delta}$ for all $N$ large enough.

Viceversa, if $|\sC_N|\ge (1+\eps)N/\sqrt{e\Delta}$ for some $\eps>0$,
there exists $\oxi(\eps)>0$ and a local algorithm that
outputs a set of vertices $\hsC_N$ satisfying
$\E[|\sC_N\triangle \hsC_N|]\le N\exp(-\oxi(\eps)\sqrt{\Delta})$ for all $N$
large enough.
\end{theorem}
Notice that, on a bounded degree graph, the hidden set $\sC_N$ can not
be identified exactly with high probability. Indeed we would not be
able to  assign  a single vertex $i$ with high
probability of success,  even if we knew exactly the status of
all of its neighbors.
On the other hand, purely random guessing  yields
$\E[|\sC_N\triangle \hsC_N|] =  N\Theta(1/\sqrt{\Delta})$.
 The last theorem, thus, establishes a threshold
behavior: local algorithms can reconstruct the hidden set with
small error if and only if $|\sC_N|$ is larger than $N/\sqrt{e\Delta}$.

Unfortunately Theorem \ref{thm:Sparse} only covers the case of sparse or locally tree-like
graphs. We let $N\to\infty$ at $\Delta$ fixed and then take $\Delta$
arbitrarily large. However, if we naively apply it to the case of
complete background graphs $G_N=K_N$, by setting $\Delta=N-2$, we get
a threshold at $|\sC_N|\approx\sqrt{N/e}$ which coincides with the one
in Theorem \ref{thm:MainClique}. 
This suggests that $\sqrt{N/e}$ might
be a fundamental limit for solving the hidden clique problem in
nearly linear time. It would be of much interest to clarify whether this is
indeed the case.

\vspace{0.2cm}

The contributions of this paper can be summarized as follows:
\begin{enumerate}
\item We develop a new algorithm based on the belief propagation
  heuristic in machine learning, that applies to the general hidden
  set problem.
\item We establish a sharp analysis of the algorithm evolution,
  rigorously establishing that it can be used to find hidden cliques of size $\sqrt{N/e}$ in
  random graphs. The analysis applies to more general noise models as
  well.
\item We generalize the hidden set problem to arbitrary graphs. For 
locally tree-like graphs of degree $(\Delta+1)$, we prove that local
algorithms succeed in finding the hidden set (up to small errors) if
and only if its size is larger than $N/\sqrt{e\Delta}$.
\end{enumerate}
The complete graph case is treated in Section \ref{sec:Complete}, with
technical proofs deferred to Section \ref{sec:CompleteProofs}. The locally
tree-like case is instead discussed in Section \ref{sec:Sparse} with proofs
in Section \ref{sec:TreeLemmas}.
%
%
\subsection{Further related work}

A rich line of research in statistics addresses the
problem of identifying the non-zero entries in a sparse vector (or
matrix) $x$ from observations $W=x+Z$ where $Z$ has typically i.i.d. standard
Gaussian entries. In particular 
\cite{arias2005near,addario2010combinatorial,arias2011detection,bhamidi2012energy} study  cases in
which the sparsity pattern of $x$ is `structured'. For instance, we can
take  $x\in\reals^{N\times N}$ a matrix with $x_{ij}=\mu$ if
$\{i,j\}\subseteq\sC_N$ and $x_{ij}=0$ otherwise.
This fits the framework studied in this paper, for $G_N$ the complete
graph and $Q_0=\normal(0,1)$, $Q_1=\normal(\mu,1)$. This literature
however disregards computational considerations. 
Greedy search methods were developed in several papers, see e.g. 
\cite{sun2008size,shabalin2009finding}.

Also, the decomposition (\ref{eq:RankDeformation}) indicates a
connection with sparse principal component analysis
\cite{zou2006sparse,johnstone2009consistency,d2007direct,d2008optimal}. 
This is the
problem of  finding a sparse low-rank approximation of a given data
matrix $W$. Remarkably, even for sparse PCA, there is a large gap
between what is statistically feasible and what is achievable by
practical algorithms. Berthet and Rigollet \cite{berthet2013computational} recently
investigated the implications of the assumption that hidden clique is
hard to solve for $|\sC_N=o(\sqrt{N})$ on sparse PCA. 

The algorithm we introduce for the case $G_N=K_N$ is analogous to the 
`linearized BP' algorithm of
\cite{montanari2006analysis,guo2006asymptotic},
and to the approximate message passing (AMP) algorithm of
\cite{DMM09,BM-MPCS-2011,BM-Universality}. These ideas have been applied to low-rank
approximation in \cite{rangan2012iterative}.  The present setting
poses however several technical challenges with respect to earlier
work in this area: $(i)$~The entries of the data matrix are not i.i.d.; $(ii)$~They
are non-Gaussian with --in general-- non-zero mean; 
$(iii)$~We seek exact recovery instead of estimation; $(iv)$~The sparsity set
$\sC_N$ to be reconstructed scales sublinearly with $N$.

Finally, let us mention that a substantial literature studies the
behavior of message passing algorithms on sparse random graphs
\cite{RiU08,MezardMontanari}. In this paper, a large part of our
technical effort is instead devoted to a similar analysis on the
complete graph, in which  simple local convergence arguments fail.
%

%
%
\subsection{Notations}

Throughout the paper, $[M]=\{1,2,\dots,M\}$ denotes the set of first
$M$ integers.  We employ a slight abuse of notation to write $[N]\bs i, j$
for $[N]\bs\{i, j\}$. The indicator function is denoted by $\ind(\,\cdot\,)$.

We write $X\sim P$ when a random variable $X$ has a distribution $P$.
We will sometimes write $\E_P$ to denote expectation with respect to
the probability distribution $P$. Probability and expectation will
otherwise be denoted by $\prob$ and $\E$.
For $a\in\reals$, $b\in\reals_+$, $\normal(a,b)$ denotes the Gaussian
distribution with mean $a$ and variance $b$. 
The cumulative distribution function of a standard Gaussian will be
denoted by $\Phi(x) \equiv \int_{-\infty}^xe^{-z^2/2}\de z/\sqrt{2\pi}$.

Unless otherwise specified, we assume all edges in the graphs mentioned
are undirected. We denote by $\di$ the neighborhood of 
vertex $i$ in a graph. 

We will often use the phrase ``for $i \in \sC_N$'' when stating certain
results. More precisely, this means that for each $N$ we are choosing
an index $i_N\in\sC_N$, which does not depend on the edge labels $W$.

We use $c$, $c_0$, $c_1$, $\dots$ and $C_1, C_2, \dots$ to denote constants independent of
$N$ and $|\sC_N|$.

Throughout, for any random variable $Z$ we will indicate by $P_{Z}$ its
law.
%
%
\section{The complete graph case: Algorithm and analysis}
\label{sec:Complete}

In this section we consider the case in which the background graph is 
complete, i.e. $G_N=K_N$. Since $G_N$ does not play any role in this case,
we shall omit all reference to it. We will discuss the reconstruction
algorithm and its analysis, and finally state a generalization of
Theorem \ref{thm:MainClique}  to the case of general distributions
$Q_0$, $Q_1$.
%
%
\subsection{Definitions}

In the present case the data consists of a symmetric matrix $W \in
\reals^{N\times N}$, with $(W_{ij})_{i<j}$ 
generated independently as follows. For an unknown set $\sC_N\subset [N]$ we have
$W_{ij} \sim Q_1$ if $\{i, j\} \subseteq \sC_N$, and $W_{ij}\sim Q_0$ otherwise. Here
$Q_1$ and $Q_0$ are distinct probability measures. We make the following assumptions:
\begin{enumerate}\label{assm1}
	\item[I.] $Q_0$ has zero mean and $Q_1$ has non-zero mean
          $\lambda$.
Without loss of generality we shall further assume that $Q_0$ has unit
variance, and that $\lambda>0$.
	\item[II.] $Q_0$ and $Q_1$ are subgaussian with common scale factor
          $\rho$.
\end{enumerate}
It will be clear from the algorithm description  that there is indeed
no loss in generality in assuming that $Q_0$ has unit variance and
that $\lambda$ is positive.
Recall that a probability distribution $P$ is subgaussian
with scale factor $\rho>0$ if for all $y \in\reals$ we have:
		\begin{align*}
			\E_P\left(e^{y(X-\E_PX)}\right) \le e^{\rho y^2/2}.
		\end{align*}
There is no loss of generality in assuming a common scale factor for
$Q_0$ and $Q_1$. 
 
The task is to identify the set $\sC_N$ from a
realization of the matrix $W$. As discussed in the introduction, the
relevant scaling is $|\sC_N|=\Theta(\sqrt{N})$ and we shall therefore
define $\kappa_N\equiv |\sC_N|/\sqrt{N}$. Further, throughout this
section, we will make use of the normalized matrix
\begin{align}
A \equiv \frac{1}{\sqrt{N}}\, W\, .
\end{align}

In several technical steps of our analysis we shall consider a sequence of instances
$\{(W_{N\times N},\sC_N)\}_{N\ge 1}$ indexed by the dimension $N$,
such that  $\lim_{N\to\infty}\kappa_N=\kappa\in(0,\infty)$. This
technical assumption will be removed in the proof of our main theorem.
%
%

\subsection{Message passing and state evolution}

The key innovation of our approach is the construction and analysis of a message passing 
algorithm that allows us to identify the hidden set $\sC_N$. 
As we demonstrate in Section \ref{sec:Sparse}, this algorithm can be
derived from \emph{belief propagation} in machine learning. However
this derivation is not necessary and the treatment here will be
self-contained. 

The message passing algorithm is iterative and at each step
$t\in\{1,2,3,\dots\}$ produces an $N\times N$ matrix $\theta^t$
whose entry $(i, j)$ will be denoted as $\theta^t_{i\to j}$ to
emphasize the fact that $\theta^t$ is not symmetric. By convention, we
set $\theta^t_{i\to i}=0$. The variables $\theta^t_{i\to j}$ will be
referred to as \emph{messages}, and their update rule is formally defined below.
\begin{definition}\label{def:mporbit}
Let $\theta^0\in\reals^{N\times N}$ be an initial condition for the
messages and, for each $t$, let $f(\,\cdot\,;t):\reals\to\reals$ be a
scalar function. The \emph{message passing orbit} corresponding to the
triple $(A, f, \theta^0)$ is the sequence of  $\{\theta^t\}_{t\ge0}$,
$\theta^t\in\reals^{N\times N}$ defined by letting, for each $t\ge 0$: 
	\begin{align} \label{eq:mp1}
		\theta^{t+1}_{i\to j} &= \sum_{\ell \in [N]\bs i, j}%
		A_{\ell i} f(\theta^{t}_{\ell\to i}, t)\,
                ,\quad \forall\, j \ne i \in [N] \, .
	\end{align}
	We also define a sequence of vectors  $\{\theta^t\}_{t\ge 1}$
        with $\theta^t = (\theta^t_i)_{i\in [N]}\in\reals^N$, by letting  
	(the entries of $\theta^t$ being indexed by $i \in [N]$) given by:
	\begin{align}
		\label{eq:mp2}
		\theta^{t+1}_i &= \sum_{\ell \in [N]\bs i} A_{\ell i} f(\theta^t_{\ell\to i}, t). 
	\end{align}
\end{definition}
The functions $f(\,\cdot\,,t)$ will be chosen so that they can be
evaluated in $O(1)$ operations.
Each iteration can be implemented with $O(N^2)$ operations. 
Indeed $(\theta^{t+1}_i)_{i\in [N]}$ can be computed in $O(N^2)$ as per
Eq.~(\ref{eq:mp2}). Subsequently we can compute $(\theta^{t+1}_{i\to
  j})_{i,j\in [N]}$ in $O(N^2)$ operations by noting that  $\theta^{t+1}_{i\to j} =
\theta^{t+1}_i-A_{ij}f(\theta^t_{j\to i},t)$. 

The proper choice of the functions $f(\,\cdot\,,t)$ plays a crucial
role in the achieving the claimed tradeoff between $|\sC_N|$ and
$N$. This choice will be optimized on the basis of the general
analysis developed below.

Before proceeding, it is useful to discuss briefly the intuition
behind the update rule introduced in Definition \ref{def:mporbit}. 
For each vertex $i$, the message $\theta^t_{i\to j}$ and the value
$\theta^t_i$ are estimates of the likelihood that $i\in\sC_N$:
they are larger for vertices that are more likely to belong to the set $\sC_N$.
In order to develop some intuition on  Definition \ref{def:mporbit}, consider a conceptually simpler
iteration operating as follows on variables $\vartheta^t = (\vartheta_i^t)_{i\in
  [N]}$. For each $i\in [N]$ we let $\vartheta^{t+1}_i = \sum_{j\in
  [N]} A_{ij}f(\vartheta^t_i;t)$. In the special case
$f(\vartheta;t) = \vartheta$ we obtain the iteration  $\vartheta^{t+1}
= A\,\vartheta^t$ which is simply the power method for computing
the principal eigenvector of $A$. As discussed in the introduction,
this does not use in any way the information that $|\sC_N|$ is much
smaller than $N$. We can exploit this information by taking
$f(\vartheta;t)$  a rapidly increasing function of $\vartheta$ that 
effectively selects the vertices $i\in [N]$ with $\vartheta_i^t$
large. We will see that this is indeed what happens within our analysis.

An important feature of the message passing version (operating on
messages $\theta_{i\to j}^t$) is that it admits a characterization that is
asymptotically exact as $N\to\infty$. 
In the large $N$ limit, the messages $\theta^t_i$ (for fixed $t$) 
converge in distribution to Gaussian random variables with certain 
mean and variance. 
In order to state this result formally, we introduce the sequence
of mean and variance parameters $\{(\mu_t,\tau_t^2)\}_{t\ge 0}$ by
letting  $\mu_0 = 1$, $\tau^2_0 = 0$ and defining , for $t\ge 0$,
\begin{align}
  \mu_{t+1} &= \lambda\kappa\,\E[f(\mu_t + \tau_t\, Z, t)] \label{eq:stateevol1}\\
  \tau^2_{t+1} &= \E[f(\tau_tZ, t)^2]\label{eq:stateevol2},
\end{align}
Here expectation is with respect to $Z \sim \normal(0, 1)$. We will refer
to this recursion as to \emph{state evolution}.
\begin{lemma}\label{lem:main}	\label{prop:conv}
Let $f(u, t)$ be, for each $t\in \naturals$ a finite-degree
polynomial. For each $N$, let $W\in\reals^{N\times N}$ be a symmetric
matrix distributed as per the model introduced above with
$\kappa_N\equiv |\sC_N|/\sqrt{N}\to \kappa\in(0,\infty)$.  Set
$\theta^0_{i\to j} =1$ and denote the associated message passing orbit
by $\{\theta^t\}_{t\ge 0}$. 

Then, for any bounded Lipschitz function $\psi:\reals\mapsto\reals$,
the following limits hold  in probability:
\begin{align}
		\lim_{N\to\infty}\frac{1}{|\sC_N|}\sum_{i\in \sC_N}
		\psi(\theta^t_i) &= \E[\psi(\mu_t +\tau_t\,  Z)]\label{eq:thetaconv1} \\
		\lim_{N\to\infty}\frac{1}{N}\sum_{i\in [N]\bs\sC_N}\psi(\theta^t_i) &= \E[\psi(\tau_t\,Z)]\label{eq:thetaconv2}. 
\end{align}
	Here expectation is with respect to $Z \sim \normal(0, 1)$ where $\mu_t,\tau_t^2$ are given by the recursion
	in Eqs. \eqref{eq:stateevol1},\eqref{eq:stateevol2}.
\end{lemma}
The proof of this Lemma is deferred to Section
\ref{sec:ProofStateEvolution}.  Naively, one would like to use the
central limit theorem to approximate the distribution on the
right-hand side of Eq.~(\ref{eq:mp1}) or of Eq.~(\ref{eq:mp2}) by a Gaussian. This is, however,
incorrect because the messages $\theta_{\ell\to i}^t$ depend on the matrix
$A$ and hence the summands are not independent. In fact, the lemma would
be false if we did not use the edge messages and replaced $\theta_{\ell\to i}^t$
by $\theta_{\ell}^t$ in Eq.~(\ref{eq:mp1}) or in Eq.~(\ref{eq:mp2}).
 
However, for the iteration Eq.~(\ref{eq:mp1}), we prove that the
distribution of $\theta^t$ is approximately the same that we would
obtain by using a fresh independent copy of $A$ (given $\sC_N$) at each iteration. The
central limit theorem then can be applied to this modified iteration. 
In order to prove that this approximation, we use the moment method, representing
$\theta^t_{i\to j}$ and $\theta^t_{i}$ as polynomials in the entries
of $A$. We then show that the only terms that survive in these
polynomials as $N\to\infty$ are the monomials which are of degree $0$,
$1$ or $2$ in each entry of $A$.

%
%
\subsection{Analysis of state evolution}

Lemma \ref{lem:main} implies that the distribution of $\theta^t_i$ is
very different depending whether $i\in \sC_N$ or not. If $i\in\sC_N$
then $\theta^t_i$ is approximately $\normal(0,\tau_t^2)$. If instead $i\in\sC_N$
then $\theta^t_i$ is approximately $\normal(\mu_t,\tau_t^2)$.

Assume that, for some choice of the functions $f(\cdot,\cdot)$
and some $t$, $\mu_t$ is positive and much larger than $\tau_t$. We can then hope to
estimate $\sC_N$  by selecting the indices $i$ such that $\theta^t_i$
is above a certain threshold\footnote{The problem is somewhat more
  subtle because $|\sC_N|\ll N$, see next section.}.  This motivates the following result.
\begin{lemma}\label{lem:approxstateevol}
  Assume that $\lambda\kappa > e^{-1/2}$. 
  Inductively define:
  \begin{align}
    p(z, \ell) = \frac{1}{\hat L_\ell}\sum_{k = 0}^{d^*}
    \frac{\hat\mu_\ell^k z^k}{k!} \, ,\;\;\;\;\;\;\;
  \hat\mu_{\ell+1} = \E[p(\hat\mu_\ell + Z, \ell)],\label{eq:OptimalPolynomials}
  \end{align}
  where $Z\sim\normal(0, 1)$ and the recursion is initialized with
  $p(z, 0) = 1$.  Here $\hat L_\ell$ is a normalization defined, for
  all $\ell\ge 1$, by
  $\hat L_\ell^2 =  \E\Big[ \left( \sum_{k=0}^{d^*} (\hat\mu_\ell Z)^k/k! \right)^2 \Big]$.
  
  Then, for any $M$ finite there exists $d^*$, $t^*$ finite such that 
  $\hat\mu_{t^*} > M$.

By setting $f(\,\cdot\,,t) = p(\,\cdot,t)$ in
the state evolution equations \eqref{eq:stateevol1} and
\eqref{eq:stateevol2} we obtain $\mu_t=\hat\mu_t$ and $\tau_t=1$ for
all $t$.
\end{lemma}
The proof of this lemma is deferred to Section \ref{sec:CompleteProofs}. Also, the
proof clarifies that setting $f(\cdot\,,t) = p(\,\cdot,t)$ is the
optimal choice for our message passing algorithm.

The basic
intuition is as follows. Consider the state evolution equations (\ref{eq:stateevol1}) and
(\ref{eq:stateevol2}). Since we are only interested in maximizing the
signal-to-noise ratio $\mu_t/\tau_t$ we can always normalize
$f(\,\cdot\,,t)$ as to have $\tau_{t+1}=1$. Denoting by
$g(\,\cdot\,,t)$ the un-normalized function, we thus have the
recursion
\begin{align*}
\mu_{t+1} =
\lambda\kappa\frac{\E[g(\mu_t+Z,t)]}{\E[g(Z,t)^2]^{1/2}}\, .
\end{align*}
We want to choose $g(\,\cdot\,t)$ as to maximize the right-hand
side.  It is a simple exercise of calculus to show that this happens
for $g(z,t) = e^{\mu_t z}$. For this choice we obtain the iteration 
$\mu_{t+1} =\lambda\kappa\,e^{\mu_t^2/2}$ that diverges to $+\infty$
if and only if $\lambda\kappa>e^{-1/2}$. Unfortunately the resulting
$f$ is not a polynomial and is therefore not covered by Lemma
\ref{lem:main}.
Lemma \ref{lem:approxstateevol} deals with this problem by approximating the function
$e^{\mu_tz}$ with a polynomial.

%
%

\subsection{The whole algorithm and general result}

As discussed above, after $t$ iterations of the message passing
algorithm
we obtain a vector $(\theta_i^t)_{i\in [N]}$  wherein for each $i$, $\theta^t_i$
estimates the likelihood that $i\in\sC_N$. We can therefore select a
`candidate' subset for $\sC_N$, by letting $\tsC_N \equiv \{
i\in[N]:\; \theta^t_i\ge \mu_t/2\}$ (this choice is motivated by the
analysis of the previous section).
Since however $\theta_i^t$ is approximately $\normal(0,\tau_t^2)$ for
$i\in[N]\setminus\sC_N$, this produces a set of size
$|\tsC_N|=\Theta(N)$, much larger than the target $\sC_N$.

\begin{algorithm}
  \caption{Message Passing}
  \label{alg:amp}
  \begin{algorithmic}[1]
    \State Initialize: $A(N) = W(N)/\sqrt{N}$; $\theta^0_{i} = 1$ for each $i\in [N]$; $d^*, t^*$ positive integers, 
    $\bar\rho$ a positive constant.
    \State Define the sequence of polynomials $p(\,\cdot\,,t)$ for
    $t\in\{0,1,\dots\}$, the values $\hat\mu_t$    as per Lemma \ref{lem:approxstateevol}
    \State Run $t^*$ iterations of message passing as in Eqs.~(\ref{eq:mp1}), (\ref{eq:mp2}) with $f(\,\cdot\,, t) = p(\,\cdot\,, t)$
    \State Find the set $\tsC_N = \{i\in[N]\, : \theta^{t^*}_i \ge \hat\mu_{t^*}/2\}$.
    \State Let $A|_{\tilde \sC_N}$ be the restriction of $A$ to the rows and columns with
    index in $\tilde \sC_N$, and compute by power method its principal eigenvector $u^{**}$.
    \State  Compute $\sB_N \subseteq [N]$ of the top 
    $|\sC_N|$ entries (by absolute value) of $u^{**}$.
    \State Return $\hsC_N = \{i\in [N]: \zeta^{\sB_N}_{\bar\rho}(i) \ge \lambda/2\}$.
  \end{algorithmic} 
\end{algorithm}

In order to overcome this problem, we apply a \emph{cleaning} procedure to
reconstruct $\sC_N$ from $\tsC_N$.
Let $A|_{\tsC_N}$ be the restriction of $A$ to the rows and columns with
index in $\tsC_N$. By power iteration (i.e. by the iteration $u^{t+1} =  A|_{\tsC_N}u^t/\|A|_{\tsC_N}u^t\|_2$,
$u^t\in\reals^{\tsC_N}$, with $u^0= (1,1,\dots,1)^{\sT}$) we compute a
good approximation $u^{**}\equiv u^{t_{**}}$ of the principal eigenvector of  $A|_{\tsC_N}$. We
then let $\sB_N\subseteq [N]$, $|\sB_N|=|\sC_N|$ be the set of indices corresponding to
the $|\sC_N|$ largest entries of $u^{t_{**}}$ (in absolute value). 
    
The set $\sB_N$ has the right size and is approximately equal to
$\sC_N$. We correct the residual `mistakes'  by defining the following
score  for each vertex $i\in [N]$:
  \begin{align}\label{eq:scoredef}
    \zeta^{\sB_N}_{\brho}(i) = \sum_{j\in \sB_N} W_{ij}\ind_{\{|W_{ij}|\le \brho\}},
  \end{align}
and returning the set $\hsC_N$ of vertices with large scores, e.g. 
$\hsC_N = \{i\in [N]: \, \zeta^{\sB_N}_{\bar\rho}(i) \ge \lambda |\sB_N|/2\}$.

Note that the `cleaning' procedure is similar to the algorithm of 
\cite{alon1998finding}. The analysis is however more challenging
because we need to start from a set $\tsC_N$ that is correlated with
the matrix $A$. 
\begin{lemma}\label{pro:algproof}
Let $A=W/\sqrt{N}$ be defined as above and $\tsC_N\subseteq [N]$ be
any subset of the column indices (possibly dependent on $A$).  
Assume that it satisfies, for $\eps$ small enough, 
$|\tsC_N\cap\sC_N| \ge (1 - \eps)|\sC_N|$ and $|\tsC_N\bs\sC_N| \le
\eps |[N]\bs\sC_N|$.

Then there exists  $t_{**}=O(\log N)$ (number of iterations in the power method)
 such that the \emph{cleaning procedure} gives $\hsC_N = \sC_N$ with high probability.
\end{lemma}
The proof of this lemma can be found in Section \ref{sec:CompleteProofs} and uses
large deviation bounds on the principal eigenvalue of $A|_{\sC_N}$.

The entire algorithm is summarized in Table \ref{alg:amp}. Notice that
the power method has complexity $O(N^2)$ per iteration and  since we
only execute $O(\log N)$ iterations, its overall complexity is
$O(N^2\log N)$. Finally the scores (\ref{eq:scoredef}) can also be
computed in $O(N^2)$ operations.
Our analysis of the algorithm results in the following main result
that generalizes Theorem \ref{thm:MainClique}.
\begin{theorem}\label{thm:main}
Consider the hidden set problem on the complete graph $G_N=K_N$, and
assume that $Q_0$ and $Q_1$ are subgaussian probability distributions
with mean, respectively, $0$, and $\lambda>0$. Further assume that
$Q_0$ has unit variance.

If $\lambda|\sC_N|\ge (1+\eps)\sqrt{N/e}$ then there 
  there exists a $\brho$, $d^*$ and $t^*$ finite such that Algorithm \ref{alg:amp} returns 
  $\hsC_N = \sC_N$ with high probability on input $W$, with total
  complexity $O(N^2\log N)$.

(More explicitly, there exists
$\delta(\eps,N)$ with $\lim_{N\to\infty}\delta(\eps,N)=0$ such that
the algorithm succeeds with probability at least $1-\delta(\eps,N)$.)
\end{theorem}
\begin{remark}
The above result can be improved if $Q_0$ and $Q_1$ are known by
taking a suitable transformation of the entries $W_{ij}$. In particular,
assuming\footnote{If $Q_1$ is singular with respect to $Q_0$ the
  problem is simpler but requires a bit more care.} that $Q_1$ is absolutely continuous with respect to $Q_0$, the
optimal such transformation is obtained by setting
\begin{align*}
A_{ij} \equiv \frac{1}{\sqrt{N}}\left[\frac{\de Q_1}{\de
    Q_0}(W_{ij})-1\right]\, .
\end{align*}
Here $\de P/\de Q$ denotes the Rad\'on-Nikodym derivative of $P$ with
respect to $Q$. If the resulting $A_{ij}$ is subgaussian with scale
$\rho/N$, then our analysis above applies. Theorem \ref{thm:main}
remains unchanged, provided the parameter $\lambda$ is replaced by 
the $\ell_2$ distance between $Q_0$ and $Q_1$:
\begin{align}
\widetilde{\lambda}\equiv \left\{\int \left[\frac{\de Q_1}{\de
      Q_0}(x)-1\right]^2\,Q_0(\de x)\right\}^{1/2}\, .
\end{align}
\end{remark}

\section{Proof of Theorem \ref{thm:main} }
\label{sec:CompleteProofs}

In this section we present the proof of  Theorem \ref{thm:main} and of
the auxiliary Lemmas \ref{lem:main}, \ref{lem:approxstateevol}  and
\ref{pro:algproof}.

We begin by showing how these technical lemmas imply Theorem
\ref{thm:main}. First consider a sequence of instances with
$\lim_{N\to\infty}|\sC_N|/\sqrt{N} = \lim_{N\to\infty}\kappa_N
=\kappa$ such that $\kappa\lambda<1/\sqrt{e}$. We will prove that 
Algorithm \ref{alg:amp} returns $\hsC_N = \sC_N$ with probability
converging to one as $N\to\infty$.

By Lemma \ref{lem:main}, we have, in probability
\begin{align*}
\lim_{N\to\infty} \frac{|\tsC_N\cap\sC_N|}{|\sC_N|} =\lim_{N\to\infty} \frac{1}{|\sC_N|}
\sum_{i\in\sC_N}
\ind(\theta^{t^*}_i\ge \hmu_{t^*}/2) = \E\{\ind(\mu_{t^*}+\tau_tZ\ge
\hmu_{t^*}/2)\}\, .
\end{align*}
Notice that, in the second step, we applied Lemma \ref{lem:main}, to
the function $\psi(z) \equiv \ind(x\ge \hmu_{t^*}/2)$. While this is
not Lipschitz continuous, it can approximated from above and below
pointwise by Lipschitz continuous functions. This is sufficient to
obtain the claimed convergence  as in standard weak
convergence arguments \cite{billingsley2008probability}. 

Since we used $f(\,\cdot\,,t) = p(\,\cdot\,,t)$, we have, by Lemma
\ref{lem:approxstateevol}, $\mu_t=\hmu_t$ and $\tau_t=1$. 
Denoting by $\Phi(z)\equiv \int_{-\infty}^ze^{-x^2/2}\de
x/\sqrt{2\pi}$ the Gaussian distribution function, we thus have
\begin{align*}
\lim_{N\to\infty} \frac{|\tsC_N\cap\sC_N|}{|\sC_N|} =
1-\Phi(-\hmu_t^*/2)\ge 1-e^{-M^2/8}\, .
\end{align*}
where in the last step we used $\Phi(-a)\le e^{-a^2/2}$ for $a\ge 0$
and  Lemma \ref{lem:approxstateevol}. By taking $M^2\ge 8\log(2/\eps)$ 
we can ensure that the last expression is larger than $(1-\eps/2)$ and
therefore $|\tsC_N\cap\sC_N|\ge (1-\eps)|\sC_N|$ with high
probability.

By a similar argument  we have, in probability
\begin{align*}
\lim_{N\to\infty} \frac{|\tsC_N|}{N} = \Phi(-\hmu_t^*/2) \le
\frac{\eps}{2}\, ,
\end{align*}
and hence $|\tsC_N|\le \eps|[N]\setminus\sC_N|$ with high probability.

We can therefore apply Lemma \ref{pro:algproof} and conclude that
Algorithm \ref{alg:amp} succeeds with high probability for
$\kappa_N=|\sC_N|/\sqrt{N}\to \kappa>1/\sqrt{\lambda^2 e}$.

In order to complete the proof, we need to prove that the Algorithm
\ref{alg:amp}  succeeds with probability at least $1-\delta(\eps,N)$
for  all $|\sC_N|\ge (1+\eps)\sqrt{N/(\lambda^2e)}$. Notice that,
without loss of generality we can assume $\kappa_N\in
[(1+\eps)/\sqrt{\lambda^2 e}, K]$ with $K$ a large enough constant
(because for $\kappa_N>K$ the problem becomes easier and --for
instance-- the proof of
\cite{alon1998finding} already works). If the claim was false there
would be a sequence of values $\{\kappa_N\}_{N\ge 1}$ indexed by $N$ such that the success
probability  remains bounded away from one along the sequence. But
since $[(1+\eps)/\sqrt{\lambda^2 e}, K]$ is compact, this sequence has
a converging subsequence along which the success probability remains
bounded. This contradicts the above.

%
%
\subsection{Proof of Lemma \ref{lem:main}}\label{sec:ProofStateEvolution}

It is convenient to collate the assumptions we make on our problem
instances as follows.

\begin{definition}
  We say $\{A(N), \F_N, \theta^0_N\}_{\{N\ge 1\}}$ is a $(C, d)$\emph{-regular}
  sequence if:
  \begin{enumerate}
    \item For each $N$, $A(N) = W_N/\sqrt{N}$ where $W_N$ satisfies
      Assumption \ref{assm1}.
    \item For each $t\ge 0$, $f(\cdot, t)\in \F_N$ is a polynomial with
      maximum degree $d$ and coefficients bounded in absolute value by $C$.
    \item Each entry of the initial condition $\theta^0_N$ is 1.
  \end{enumerate}
\end{definition}

Let $A^t, t\ge 1$ be i.i.d.  matrices distributed as $A$ conditional on the set $\sC_N$, 
and let $A^0 \equiv A$. We now define the sequence of $N \times N$ matrices $\{\xi^t\}_{t\ge 0}$ 
and a sequence of vectors in $\reals^N$, $\{\xi^t\}_{t\ge 1}$
(indexed as before) given by:

\begin{align}
	\label{eq:rmp1}
	\xi^{t+1}_{i\to j} &= \sum_{\ell \in [N]\bs\{i, j\}}%
	A^t_{i\ell} f(\xi^{t}_{\ell\to i}, t) \\
	\xi^0_{i\to j} &= \theta^0_i \quad \forall\, j \ne i \in [N] \nonumber \\
	\xi^t_{i\to i} &= 0 \quad \forall\, t\ge0, i \in [N] \nonumber \\
	\xi^{t+1}_i &= \sum_{\ell\in [N]\bs i}A^t_{\ell i} f(\xi^t_{\ell\to i}, t) \label{eq:rmp2}
\end{align}

The asymptotic marginals of the iterates $\xi^t$ are easier to compute since the
matrix $A^{t-1}$ is independent of the $\xi^{t-1}$ by definition. We proceed, hence
by proving that $\xi^t$ and $\theta^t$ have, asymptotically in $N$, the same
moments of all orders computing the distribution for the $\xi^t$.

The messages $\theta^t_{i\to j}$ and $\xi^t_{i\to j}$ can be described explicitly via a 
sum over a family of finite rooted labeled trees. We now describe this
family in detail.
All edges are assumed directed towards the root. The leaves of the tree are those 
vertices with no children, and the set of leaves is denoted by $L(T)$. We let $V(T)$
denote the set of vertices of $T$ and $E(T)$ the set of (directed) edges in $T$. 
The root has a label in $[N]$ called its ``type''. Every non-root vertex has
a label in $[N]\times\{0, 1, \ldots, d\}$, the first argument the 
label being the ``type'' of the vertex, and the second being the ``mark''. For
a vertex $v\in T$ we let $l(v)$ denote its type, $r(v)$ its mark and $|v|$ its
distance from the root in $T$.

\begin{definition}
	Let $\cT^t$ be the family of labeled trees $T$ with exactly $t$ generations 
	satisfying the conditions:
	\begin{enumerate}
		\item The root of $T$ has degree 1.
		\item Any path $v_1, v_2 \ldots v_k$ in the tree is non-backtracking
			i.e. the types $l(v_i)$, $l(v_{i+1})$, $l(v_{i+2})$ are distinct.
		\item For a vertex $u$ that is not the a root or a leaf, the mark $r(u)$ is
			set to the number of children of $v$.
		\item We have that $t = \max_{v\in L(T)} |v|$. All leaves $u \in L(T)$ with
			non-maximal depth, i.e. $|u| \le t-1$ have mark 0. 
	\end{enumerate}
	Let $\cT^t_{i\to j} \subset \cT^t$ be the subfamily satisfying, in addition, the following:
	\begin{enumerate}
		\item The type of the root is $i$. 
		\item The root has a single child with type distinct from $i$ and $j$. 
	\end{enumerate}
	In a similar fashion, let $\cT^t_i \subset \cT^t$ be the subfamily satisfying, additionally:
	\begin{enumerate}
		\item The type of the root is $i$. 
		\item The root has a single child with type distinct from $i$. 
	\end{enumerate}
\end{definition}

Let the polynomial 
$f(x, t)$ be represented as:
\begin{align*}
	f(x, t) &= \sum_{i=0}^d q^t_i x^i
\end{align*}

For a labeled tree $T \in \cT^t$ and vector of coefficients $\bq = (q^s_i)_{s \le t, i \le d}$ 
we now define three weights:
\begin{align}\label{eq:treeweights}
	A(T) &\equiv \prod_{u\to v \in E(T)} A_{l(u)l(v)} \\
	\Gamma(T, \bq, t) &\equiv \prod_{u\to v \in E(T)} q^{t - |u|}_{r(u)} \\
	\theta(T) &\equiv \prod_{u\in L(T)} (\theta^0_{l(u)})^{r(u)}
\end{align}
We now are in a position to provide an explicit expression for $\theta^t_{i\to j}$ in terms
of a summation over an appropriate family of labeled trees. 
\begin{lemma}
  Let $\{A(N), \F_N, \theta^0_N\}$ be a $(C, d)$-regular sequence. The orbit $\theta^t$ satisfies:
	\begin{align}
		\theta^t_{i\to j} &= \sum_{T \in \cT^t_{i\to j}} A(T)\Gamma(T, \bq, t)\theta(T) \label{eq:treerep1} \\
		\theta^t_i &= \sum_{T \in \cT^t_{i}} A(T)\Gamma(T, \bq, t)\theta(T) \label{eq:treerep2}
	\end{align}
	\label{lem:treerep}
\end{lemma}
\begin{proof}
	We prove \myeqref{eq:treerep1} using induction. The proof of \myeqref{eq:treerep2}
	is very similar. We have, by definition, that: 
	\begin{align*}
		\theta^1_{i\to j} &= \sum_{\ell\in[N]\bs i, j} \sum_{k\le d} A_{\ell i}q^0_k(\theta^0_\ell)^k
	\end{align*}
	This is what is given by \myeqref{eq:treerep1} since $\cT^1_{i\to j}$ 
	is exactly the set of trees with two vertices joined by a single edge, the root 
	having type $i$, the other vertex (say $v$) having type $l(v) \notin \{i, j\}$ 
	and mark $r(v) \le d$. 

	Now we assume \myeqref{eq:treerep1} to be true up to $t$. For iteration $t+1$, 
	we obtain by definition:
	\begin{align*}
		\theta^{t+1}_{i\to j} &= \sum_{\ell\in [N]\backslash\{i, j\}} A_{\ell i}%
		\sum_{k \le d} q^t_k (\theta^t_{\ell\to i})^k\\
		&= \sum_{\ell\in[N]\backslash\{i, j\}} \sum_{k\le d}\sum_{T_1 \cdots T_k \in \cT^t_{\ell\to i}} %
		A_{\ell i}q^t_k\prod_{m =1}^k A(T_m) \Gamma(T_m, \bq, t)\theta(T_m)
	\end{align*}
	Notice that $\cT^{t+1}_{i\to j}$ is in bijection with the set of pairs
	containing a vertex of type $\ell \notin \{i, j\}$ and a $k$-tuple of trees belonging
	to $\cT^t_{\ell\to i}$. This is because one can form a tree in $\cT^{t+1}_{i\to j}$
	by choosing a root with type $i$, its child $v$ with type $\ell \notin \{i, j\}$
	and choosing a $k (\le d)$-tuple of trees from $\cT^t_{\ell\to i}$, identifying their
	roots with $v$ and setting $r(v) = k$. With this, absorbing the factors of $A_{\ell i}$
	into $\prod_{m=1}^k A(T_m)$ and $q^t_k$ into $\prod_{m=1}^k \Gamma(T_m, \bq, t)$ yields
	the desired claim.
\end{proof}

From a very similar argument as above we obtain that:
\begin{align*}
	\xi^t_{i\to j} &= \sum_{T\in\cT^t_{i\to j}} \bar A(T)\Gamma(T, \bq, t) \theta(T) \\
	\xi^t_i &= \sum_{T\in\cT^t_i} \bar A(T)\Gamma(T, \bq, t) \theta(T)
\end{align*}
where the weight $\bar A(T)$ for a labeled tree $T$ is defined (similar
to \myeqref{eq:treeweights}) by:
\begin{align}\label{eq:reftreeweight}
	\bar A(T) &\equiv \prod_{u\to v \in E(T)} A^{t - |u|}_{l(u)l(v)}  
\end{align}

We now prove that the moments of $\theta^t_{i}$ and $\xi^t_i$ are asymptotically
(in the large $N$ limit) the same via the following:
\begin{proposition}
	Let $\{A(N), \F_N, \theta^0_N\}$ be a $(C, d)$-regular sequence.	
	the conditions above. Then, for any $t \ge 1$, there exists a constant $K$
	independent of $N$ (depending possibly on $m, t, d, C$) such that for any 
	$i\in[N]$:
	\begin{align*}
		\left\lvert \E\left[(\theta^t_i)^m\right] - \E\left[ (\xi^t_i)^m \right]\right\rvert &\le KN^{-1/2} 
	\end{align*}
	\label{prop:refresh}
\end{proposition}
\begin{proof}
	According to our initial condition, $\theta^{0, N}$ has all entries 1. 
	Then, using the tree representation we have that:
	\begin{align}
		\E \left[ (\theta^t_i)^m \right] &= \sum_{T_1,\ldots,T_m\in\cT^t_i}\left[\prod_{\ell=1}^m \Gamma(T_\ell, \bq, t)\right]%
		\E\left[ \prod_{\ell=1}^m A(T_\ell) \right] \label{eq:mom1}\\
		\E \left[ (\xi^t_i)^m \right] &= \sum_{T_1,\ldots,T_m\in\cT^t_i}\left[\prod_{\ell=1}^m \Gamma(T_\ell, \bq, t)\right]%
		\E\left[ \prod_{\ell=1}^m \bar A(T_\ell) \right] \label{eq:mom2} 
	\end{align}
	
	Define the multiplicity $\phi(T)_{rs}$ to be the number of occurrences of an edge $u\to v$ in the
	tree $T$ with types $l(u), l(v) \in \{r, s\}$.	
	Also let $\bG$ denote the graph obtained by identifying vertices of the same type in the tuple of
	trees $T_1, \ldots T_m$. We let $\Gs$ denote its restriction to 
	the vertices in $\sC_N$ and $\Gsc$ be the graph restricted to $\sC_N^c$.  
	Let $E(\Gs)$ and $E(\Gsc)$ denote the (disjoint) edge sets of these graphs and $E_J$ denote
	the edges in $\bG$ not present in either $\Gs$ or $\Gsc$. In other words, $E_J$
	consists of all edges in $\bG$ with one endpoint belonging to $\sC_N$ and one end point
	outside it. The edge sets here do not count multiplicity. 
	
	For analysis, we first split 
	the sum over $m$-tuples of trees above into three terms as follows:
	
	\begin{enumerate}
		\item $S(A)$: the sum over all $m$-tuples of trees $T_1, \ldots, T_m$ such that
			there exists an edge $rs$ in $E(\Gsc)\cup E_J$ which is covered at least
			3 times.
		\item $R(A)$: the sum over all $m$-tuples of trees such that each edge in $E(\Gsc)\cup E_J$ 
			is covered either 0 or 2 times, and the graph $\bG$ contains a cycle.
		\item $T(A)$: the sum over all $m$-tuples of trees such that each edge in $E(\Gsc)\cup E_J$
			is covered either 0 or 2 times, and the graph $\bG$ is a tree.
	\end{enumerate}
	
	We also define analogous terms $S(\bar A)$, $R(\bar A)$ and $T(\bar A)$ in the same fashion.
	We have that $\left\lvert\prod_{\ell=1}^m\Gamma(T_\ell, \bq, t)\right\rvert \le C^{md^{t+1}}$
	since the coefficients are bounded by $C$ and the number of edges in the tree by $d^{t+1}$.
	We thus concentrate on the portion $\E\left[ \prod_{\ell=1}^m A(T_\ell) \right]$. 
	When $\E\left[ \prod_{\ell=1}^m A(T_\ell) \right] = 0$, some edge in $E(\Gsc)\cup E_J$
	is covered exactly once. This implies $\E\left[ \prod_{\ell=1}^m \bar A(T_\ell) \right]=0 =%
	\E\left[ \prod_{\ell=1}^m A(T_\ell) \right]$,
	since the same edge is covered only once in any generation. This guarantees
	that we need only consider the contributions $S(A)$, $R(A)$ and $T(A)$ as above in the sums
	\myeqref{eq:mom1}, \myeqref{eq:mom2}.
	
	We first consider the contribution $S(A)$. We have:
	\begin{align}
		\E \left[  \prod_{\ell=1}^m A(T_\ell)\right] &=\E\Bigg[ \prod_{j<k} (A_{jk})^%
		{\sum_{\ell = 1}^m \phi(T_\ell)_{jk}}  \Bigg] \nonumber \\
		&\le \E\left[ \prod_{j<k} |A_{jk}|^{\sum_{\ell=1}^m\phi(T_\ell)_{jk}} \right] \nonumber \\
		&=\prod_{j<k}\E\left[|A_{jk}|^{\sum_{\ell=1}^m\phi(T_\ell)_{jk}} \right] \nonumber \\
		&\le C_1\left( \frac{1}{\sqrt N} \right)^\alpha \label{eq:mombndA},
	\end{align}
	where $\alpha = \alpha(T_1,\ldots, T_m)$ is the total number of edges (with multiplicity) 
	in the tuple of trees $T_1, \ldots, T_m$. The last inequality follows from Lemma \ref{lem:subgauss}
	and observing that for any $j, k$, $A_{jk}$ is subgaussian with scale parameter $\rho/N$.  
	The constant $C_1 = C_1(T_1, \ldots T_m)$ absorbs the leading factors from Lemma \ref{lem:subgauss}, and is
	independent of $N$.

	To track the dependence on $N$, note that the
	graph $\bG$ is connected since the roots of all the trees have type $i$. Let $n(\Gs)$ [resp. $n(\Gsc)$] denote 
	the number of vertices in $\Gs$ [resp. $\Gsc$] not counting the root.  
	Counting the edges with multiplicities, we have $3 + |E(\Gs)| + 2(|E(\Gsc)| + |E_J| -1)%
	\le \alpha$, implying $|E(\Gs)| + 2(|E(\Gsc)|+|E_J|)%
	\le \alpha-1$. By connectivity of $\bG$ and the fact that each component in $\Gsc$ is
	connected by at least one edge to a vertex of $\Gs$ we have that $n(\Gsc) + n(\Gs) \le%
	|E_J| + |E(\Gsc)| + |E(\Gs)|$ and $n(\Gsc) \le |E_J| + |E(\Gsc)|$. Combining we get: 
	\begin{align*}
		n(\Gs) + 2n(\Gsc) &\le |E(\Gs)| + 2(|E(\Gsc)| + |E_J|) \\
		&\le \alpha - 1
	\end{align*}
	For a candidate graph $\bG$, the number of possible labels of types is upper bounded by
	\begin{align*}
		\text{no. of possible labeling of }\bG  &\le 2^{n(\Gs) + n(\Gsc)}(\kappa_N\sqrt N)^{n(\Gs)}(N)^{n(\Gsc)} \\
		&\le  (4\kappa\sqrt N)^{\alpha-1}, 
	\end{align*}
	for large enough $N$.
	Denote by $\cU^t_i$ the set of trees $\cT^t_i$ with the labels removed.
	We then have, using the above and \myeqref{eq:mombndA}
	\begin{align}
		|S(A)| &\le  C^{md^{t+1}}\sum_{(\cU^t_i)^m} C_1 (\sqrt N)^{-\alpha }(4\kappa\sqrt N)^{\alpha - 1}\nonumber\\
		&\le C_2 N^{-1/2} \label{eq:contribedgecover},
	\end{align}
	where we absorbed the summation over $(\cU^t_i)^m$ into $C_2$ since it is independent of $N$.
	The constant $C_1$ appears because the same tuple of (unlabeled) trees can yield
	different (candidate) graphs $\bG$, however their total number is independent of $N$.

	Indeed, we can do a similar calculation to obtain that $|R(A)|$ is $O(N^{-1/2})$.  
	For such a graph, $n(\Gs) + n(\Gsc) = |E(\Gs)| + |E(\Gsc)| + |E_J| - a$ for some $a \ge 1$ 
	when $\bG$ has at least one cycle. We have $|E(\Gs)| + 2(|E(\Gsc)| + |E_J|) \le \alpha$
	by counting minimum multiplicities and $n(\Gsc) \le |E(\Gsc)| + |E_J|$ by connectivity
	argument. Thus:
	\begin{align*}
		n(\Gs) + 2n(\Gsc) &\le \alpha - a \\
		&\le \alpha -1.
	\end{align*}
	The number of possible labels for $\bG$ is thus bounded above by $(4\kappa\sqrt N)^{\alpha-1}$.
	Following the same argument as before, we get:
	\begin{align}
		|R(A)| \le C_3 N^{-1/2}, \label{eq:contribGcycle}
	\end{align}
	for some constant $C_3$ dependent only on $m, d, t, \kappa$. We note here that the same bounds
	hold for $S(\bar A)$ and $R(\bar A)$. Indeed, let $\varphi(T)_{rs}^g$ denote the number of
	times an edge $u\to v$ of (distinct) types $l(u), l(v) \in \{r, s\}$ is covered with $|u|=g$. 
	By definition, it follows that $\sum_g \varphi(T)_{rs}^g = \phi(T)_{rs}$. We then obtain:
	\begin{align}
		\E \left[  \prod_{\ell=1}^m\bar A(T_\ell)\right] &=\E\Bigg[ \prod_{j<k}\prod_g (A^{g-1}_{jk})^%
		{\sum_{\ell = 1}^m \varphi(T_\ell)_{jk}^g}  \Bigg] \nonumber\\
		&\le \E\left[ \prod_{j<k} \prod_g |A_{jk}^{g-1}|^{\sum_{\ell=1}^m \sum_g\varphi(T_\ell)_{jk}^g} \right]\nonumber \\
		&= C_4\left( \frac{1}{\sqrt N} \right)^\alpha \label{eq:mombndbarA},
	\end{align}
	This can be used in place of \myeqref{eq:mombndA} to obtain the required bounds on 
	$S(\bar A)$ and $R(\bar A)$. 

	By the bounds \myeqref{eq:contribedgecover}, \myeqref{eq:contribGcycle},
	to prove our result we only need to concentrate
	on $T(A)$. It suffices to show that $T(A) = T(\bar A)$. 
	We first consider the case $\E\left[\prod_{\ell=1}^m A(T_\ell)\right] \ne 0 = \E\left[%
	\prod_{\ell=1}^m \bar A(T_\ell)\right]$. This implies that there exists an edge $rs$ in $E(\Gsc)\cup E_J$ 
	with multiplicity 2, but appearing in different generations in the tuple of trees. Suppose
	they appear on the same branch of the tree, call it $T_1$. Then there exists $a\to b$ and $c\to d$
	with $\{l(a), l(b)\} = \{l(c), l(d)\} = \{i, j\}$ with $a\to b$ on the path from $c$ to the root.
	Due to the non-backtracking property, $a\ne d$. However, then these edges form a cycle in $\bG$
	(formed of the edges from $d$ to $a$) because the tree is non-backtracking and we arrive at
	a contradiction. Now suppose the edges $a\to b$ and $c\to d$ as above appear in different
	generations in distinct trees $T_1$ and $T_2$ respectively. Then as the roots of the $T_\ell$'s
	identify to the same vertex, and the trees are non back-tracking, these form a cycle in $\bG$
	and we arrive at a contradiction. Using the same argument, we see that such edges as $a\to b$
	and $c\to d$ above cannot exist even on different branches of the same tree in different
	generations. 
	
	Now assume $\E\left[ \prod_{\ell=1}^m \bar A(T_\ell) \right] \ne 0$. This means that 
	every edge in $E(\Gsc)\cup E_J$ is covered exactly twice in the same generation and every
	edge in $E(\Gs)$ is covered at most twice. Then, if $\E\left[ \prod_{\ell=1}^m\bar A(T\ell) \right] %
	\ne \E\left[ \prod_{\ell=1}^m A(T\ell) \right]$, there must exist an edge $rs \in E(\Gs)$
	covered twice i.e. with multiplicity, but in two different generations. However, by the 
	argument given previously, this is not possible. 
	We thus obtain that $T(A) = T(\bar A)$. Using this and the bounds on $S(A)$, $R(A)$
	we obtain the required result, for an appropriately adjusted leading constant $K$ depending on $m, d, t$
	and $\kappa$.

\end{proof}

Before proceeding, we prove the following results that are useful to establish state evolution.

\begin{lemma}
	Consider the situation as assumed in Lemma \ref{prop:refresh}. Then we have, for some constants
	$K_m(m, d, t, \kappa), K'_m(m, d, t, \kappa)$ independent of $N$ that:
	\begin{align*}
		|\E[(\xi^t_{i\to j})^m]| &\le K_{m} \\
		|\E(\xi^t_{i})^m| &\le K'_{m}
	\end{align*}
	\label{lem:mombnd}
\end{lemma}
\begin{proof}
	We prove the claim for $\xi^t_{i\to j}$. The other claim follows by essentially the
	same argument. Recall from the tree representation of Lemma \ref{lem:treerep}:
	\begin{align*}
		\E[(\xi^t_{i\to j})^m] &= \sum_{T_1, \ldots, T_m \in \cT^t_{i\to j}} %
		\left[\prod_{\ell=1}^m \Gamma(T_\ell, \bq, t)\right]%
		\E\left[ \prod_{\ell=1}^m \bar A(T_\ell) \right]
	\end{align*}
	Using the same splitting of contributions to the above sum into $S(\bar A)$,
	$R(\bar A)$ and $T(\bar A)$ as in Lemma \ref{prop:refresh}, we see that it is
	sufficient to prove that $|T(A)|$ is bounded uniformly over $N$. We have:
	\begin{align*}
		|T(\bar A)| &= \sum_{T_1,\ldots, T_m}\left[\prod_{\ell=1}^m \Gamma(T_\ell, \bq, t)\right]%
		\E\left[ \prod_{\ell=1}^m \bar A(T_\ell) \right] \\
		&\le \sum_{T_1,\ldots T_m} C^{d^{t+1}} C_1(\sqrt N)^\alpha,
	\end{align*}
	where $\alpha = \alpha(T_1, \ldots, T_m)$ is the number of edges counted with multiplicity
	and $T_1, \ldots T_m$ ranges over $m$-tuples of trees such that the graph $\bG$ (formed
	by identifying vertices of the same type) is a tree. Define $n(\Gs)$, $n(\Gsc)$, $E(\Gs)$,
	$E(\Gsc)$ and $E_J$ as in Lemma \ref{prop:refresh}. By an argument similar to that for
	bounding $R(A)$, we obtain that $n(\Gs) + 2n(\Gsc) = \alpha$. Thus we get:
	\begin{align*}
		|T(\bar A)| &\le C_5 (\sqrt N)^{-\alpha} (4\kappa\sqrt N)^\alpha \\
		&\le K_m,
	\end{align*}
	where $K_m = K(m, d, t, \kappa)$ is a constant independent of $N$. For convenience, we make
	only the dependence on $m$ explicit. The result follows, after
	a small change in the constant $K$ since
	the other contributions $S(\bar A)$ and $R(\bar A)$ are $O(N^{-1/2})$.
\end{proof}

\begin{lemma}
	Consider the situation as in Lemma \ref{prop:refresh}. Then we have:
	\label{lem:sumvar}
	\begin{align*}
		\lim_{N\to\infty}\Var\left(\frac{1}{|\sC_N|} \sum_{i\in\sC_N}(\xi^t_i)^m \right) &= 0 \\
		\lim_{N\to\infty}\Var\left(\frac{1}{|\sC_N|} \sum_{i\in\sC_N\bs j}(\xi^t_{i\to j})^m \right) &= 0\\
		\lim_{N\to\infty}\Var\left(\frac{1}{N} \sum_{i\in[N]\bs \sC_N}(\xi^t_i)^m \right) &= 0 \\
		\lim_{N\to\infty}\Var\left(\frac{1}{N} \sum_{i\in[N]\bs \sC_N, j}(\xi^t_{i\to j})^m \right) &= 0\\
		\lim_{N\to\infty}\Var\left(\frac{1}{N} \sum_{i\in[N]}(\xi^t_{i})^m \right) &= 0\\
		\lim_{N\to\infty}\Var\left(\frac{1}{N} \sum_{i\in[N]\bs j}(\xi^t_{i\to j})^m \right) &= 0,
	\end{align*}
	where $\Var(\cdot)$ denotes the variance of the argument. The same results
	hold with $\theta^t$ instead of $\xi^t$.
\end{lemma}
\begin{proof}
	We prove only the first claim in detail. The proofs for the rest of the claims 
	follow the same analysis. To begin with:
	\begin{align*}
		\Var\left( \frac{1}{|\sC_N|}\sum_{i\in\sC_N}(\xi^t_i)^m \right) &= \frac{1}{|\sC_N|^2} \sum_{i, j \in \sC_N} %
		\left(\E\left[ (\xi^t_i)^m(\xi^t_j)^m \right] - \E\left[ (\xi^t_i)^m \right]\E\left[ (\xi^t_j)^m \right]\right).
	\end{align*}
	Note that the terms wherein $i=j$ are $O(|\sC_N|)$, using Lemma \ref{lem:mombnd}. We now control 
	each of the remaining summands, where $i, j$ distinct, in the following fashion. Fix a pair $i, j$.
	The summand $\left(\E\left[ (\xi^t_i)^m(\xi^t_j)^m \right] - %
	\E\left[ (\xi^t_i)^m \right]\E\left[ (\xi^t_j)^m \right]\right)$
	 can be written as a summation over $2m$-tuples of trees $T_1,\ldots,T_m,$ $ T_1',\ldots,T_m'$
	where the first $m$ belong to $\cT^t_i$ and the last $m$ to $\cT^t_j$. Let $\bG$ 
	denote the simple graph obtained by identifying vertices of the same type in the 
	tuple $T_1,\ldots, T_m, T_1', \ldots, T_m'$. Let $\Gs$ and $\Gsc$ be subgraphs defined
	as in Proposition \ref{prop:refresh}. The terms in which $\bG$ is disconnected with 
	one component containing $i$ and the other containing $j$, are identical in 
	$\E\left[ (\xi^t_i)^m(\xi^t_j)^m \right]$ and $\E\left[ (\xi^t_i)^m \right]\E\left[ (\xi^t_j)^m \right]$ 
	and hence cancel each other. If $\bG$ is connected, by the argument 
	in Lemma \ref{prop:refresh} all terms where $\bG$ is not a tree, or when $\Gsc$ contains
	an edge covered thrice or more have vanishing contributions. It remains to check the
	contributions of terms where $\bG$ is a connected tree, and every edge in $\Gsc$
	is covered at exactly twice. Defining $n(\Gs)$, $n(\Gsc)$, $E(\Gs)$, $E(\Gsc)$ and
	$E_J$ as before, we have that $n(\Gs) + n(\Gsc) \le E(\Gs) + E(\Gsc) + E_J - 1$ since
	types $i$ and $j$ have been fixed. As before $n(\Gsc) \le E(\Gsc) + E_J$ by connectivity
	and $E(\Gs) + 2(E(\Gsc) + E_J) \le \alpha$ where $\alpha$ is the number of edges counted with
	multiplicity. This yields $n(\Gs) + 2n(\Gsc) \le \alpha -1$. The total number of such terms
	is thus at most $O(N^{(\alpha-1)/2})$, while their weight is bounded by $O(N^{-\alpha/2})$. 
	Their overall contribution, consequently, vanishes in limit. We thus have $\forall i\ne j \in \sC_N$:
	\begin{align*}
		\E\left[ (\xi^t_i)^m(\xi^t_j)^m \right] - \E\left[ (\xi^t_i)^m \right]\E\left[ (\xi^t_j)^m \right] &\le \eps(N),
	\end{align*}
	where $\eps(N) \to 0$ as $N\to\infty$. This gives:
	\begin{align*}
		\Var\left( \frac{1}{|\sC_N|}\sum_{i\in\sC_N}(\xi^t_i)^m \right) &\le O(|\sC_N|^{-1}) + \eps(N),
	\end{align*}
	and the first claim follows. 

	The other claims follow using the same argument, and since $|\sC_N| = o(N)$.
\end{proof}

\begin{proposition}\label{prop:stateevol}
	Let $\mu_t, \tau_t$ be given as in Eqs. \ref{eq:stateevol1}, \ref{eq:%
	stateevol2}. Consider $(A(N), \cF_N, \theta^0_N)_{N\ge1}$ a 
	sequence of  $(C, d)$-regular MP instances. 
	Then the following limits hold for each $m \ge 1$ and $t\ge 0$:
	\begin{align}
	  \lim_{N\to\infty} \E[(\theta^t_i)^m] &= \E[(\mu_t + Z_t)^m] \text{ if } i\in \sC_N \\
	  \lim_{N\to\infty} \E[(\theta^t_i)^m] &= \E[(Z_t)^m] \text{ otherwise.}
	\end{align}
	where $Z_t \sim \normal(0, \tau_t^2)$. 
\end{proposition}
\begin{proof}
	 
	Fix $j\ne i$. We prove by induction over $t$ that for all $t\ge 0$ and $m \ge 1$:
	\begin{align}
	  \lim_{N\to\infty} \E[(\xi^{t+1}_{i\to j})^m] &= \E[(\mu_{t+1} +Z_{t+1})^m] \quad\text{ if } i\in\sC_N \label{eq:xistateevol1a}\\
	  \lim_{N\to\infty} \E[(\xi^{t+1}_{i\to j})^m] &= \E[(Z_{t+1})^m] \quad\text{ otherwise } \label{eq:xistateevol1b}\\
		\lim_{N\to\infty} \frac{1}{|\sC_N|}\sum_{k\in \sC_N}(\xi^{t+1}_{k\to j})^m &= \E[(\mu_{t+1} + Z_{t+1})^m] \label{eq:xistateevol2}\\
		\lim_{N\to\infty} \frac{1}{N}\sum_{k\in[N]\bs \sC_N} (\xi^{t+1}_{k\to j})^m &= \E[(Z_{t+1})^m], \label{eq:xistateevol3}
	\end{align}
	where Eqs. \eqref{eq:xistateevol2} and \eqref{eq:xistateevol3} hold in probability.
	For $t\ge 1$, denote by $\mfF_t$ the $\sigma$-algebra generated by $A^0,\ldots,A^{t-1}$. 
For convenience of notation, we write $\tilde A^t_{ij}$ as the centered version
of $A^t_{ij}$. Hence $\tilde A_{ij} = A_{ij} - \lambda/\sqrt N$ if both $i, j \in \sC_N$, else
$\tilde A_{ij} = A_{ij}$. 
	First consider the case where index $i \in \sC_N$. Then we have, for any $j\ne i$:
	\begin{align*}
		\lim_{ N\to\infty} \E\left[\xi^{t+1}_{i\to j}\rvert \mfF_t\right] &= \lim_{N\to\infty}\E\biggl[\sum_{\ell \in \sC_N\bs j} A^t_{\ell i}f(\xi^t_{\ell\to i}, t) %
		+ \sum_{\ell\in[N]\bs \sC_N, j}A^t_{\ell i} f(\xi^t_{\ell\to i}, t) \bigg\rvert \mfF_t\biggr]\\
		&= \lambda\kappa\,\E[f(\mu_t + Z_t, t)]\quad\text{in probability}\\
		&= \mu_{t+1},
	\end{align*}
where $Z_t\sim \normal(0, \tau_t^2)$. Here the second equality follows from the induction hypothesis and
the third from definition. Considering the variance we have: 
\begin{align*}
	\lim_{ N\to\infty} \Var\left[\xi^{t+1}_{i\to j}\rvert \mfF_t\right] &= \lim_{N\to\infty}%
	\E\biggl[\sum_{\ell\in[N]\bs j}(\tilde A^t_{\ell i} f(\xi^t_{\ell\to i}, t))^2 \bigg\rvert \mfF_t\biggr]\\
	&= \lim_{N\to\infty} \frac{1}{N}\sum_{\ell\in[N]\bs j} (f(\xi^t_{\ell\to i}, t))^2\\
	&= \E[f(Z_t, t)^2] \\
	&= \tau_{t+1}^2,
\end{align*}
where the penultimate equality holds in probability, and follows from the induction hypothesis.

Notice that $[\xi^{t+1}_{i\to j}\lvert\mfF_t - \E(\xi^{t+1}_{i\to j}\lvert\mfF_t)]$ is a sum of independent
random variables (due to the conditioning on $\mfF_t$). We show that, in probability, the Lindeberg 
condition for the central limit theorem holds. By the induction hypothesis we have,
in probability:
\begin{align*}
	\lim_{N\to\infty}\frac{1}{N}\sum_{\ell\in[N]\bs j} (f(\xi^t_{\ell\to i}, t))^4 &= \E\left[ (f(Z_t, t))^4 \right] \\
\end{align*}
Using this we have, for any $\eps>0$: 
\begin{align*}
	\sum_{\ell\in[N]\bs j}\E\biggl[(\tilde A_{\ell i}f(\xi^t_{\ell\to i}, t))^2%
	\ind_{\{|\tilde A_{\ell i} f(\xi^t_{\ell\to i}, t)|\ge\eps\}}\bigg\rvert \mfF_t\biggr]%
	&\le C_6\left(\frac{\rho}{\eps N}\right)^2\sum_{\ell\in[N]\bs j} (f(\xi^t_{\ell\to i}, t))^4 \overset{p}{\to}0,
\end{align*}
using the induction hypothesis and Lemma \ref{lem:subgauss}. The constant $C_6$ here
comes from the leading factors in Lemma \ref{lem:subgauss}. It follows from Lemma 
\ref{lem:lindeberg} that for a bounded function $h:\reals\to\reals$ with bounded 
first, second and third derivatives that:
\begin{align*}
	\lim_{N\to\infty} \E[ h(\xi^{t+1}_{i\to j})\rvert \mfF_t] = \E[ h(Z_{t+1} + \mu_{t+1})] \text{ in probability.} 
\end{align*}

Since the functions $h^+_m(x) = (x^m)_+$ and $h^-_m(x) = (x^m)_-$ for $m \ge 3$ 
can be approached pointwise by a sequence of bounded functions
with bounded first, second and third derivatives and since $\E\left[ (\xi^{t+1}_{i\to j})^m\rvert\mfF_t \right]$ is
integrable by definition we have :
\begin{align*}
	\lim_{N\to\infty}\E\left[(\xi^{t+1}_{i\to j})^m\rvert\mfF_t\right] &= \E\left[(\mu_{t+1} + Z_{t+1})^m\right] \text{ in probability.}
\end{align*}
By the tower property of conditional expectation and Lemma \ref{lem:mombnd}, 
the expectations also converge yielding the induction claim 
\myeqref{eq:xistateevol1a}. Employing Chebyshev inequality and 
Lemma \ref{lem:sumvar} on the sequence 
$\left\{\sum_{k\in\sC_N}(\xi^{t+1}_{k\to j})^m / |\sC_N| \right\}_{N\ge 1}$, we obtain the induction
claim \myeqref{eq:xistateevol2}.

We now turn to the case when $i\notin \sC_N$. By \myeqref{eq:rmp1}:
\begin{align*}
	\E[\xi^{t+1}_{i\to j}\lvert\mfF_t] &= \E\biggl[ \sum_{\ell\in[N]\bs j}\tilde A^t_{\ell i} f(\xi^t_{\ell\to i}, t) \bigg\lvert\mfF_t\biggr] \\
	&= 0.
\end{align*}
For the variance we compute:
\begin{align*}
	\lim_{N\to\infty}\Var[\xi^{t+1}_{i\to j}\lvert\mfF_t] &= %
	\lim_{N\to\infty}\E\biggl[ \sum_{\ell\in[N]\bs j} (\tilde A^t_{\ell i}f(\xi^t_{\ell\to i}, t))^2\biggr]\\
	&= \lim_{N\to\infty} \frac{1}{N}\sum_{i\in[N]\bs j}(f(\xi^t_{\ell\to i}, t))^2\\
	&= \E[(f(Z_t, t))^2] \\
	&= \tau^2_{t+1}.
\end{align*}
The penultimate equality holds in probability, from the induction hypothesis and the last
equality by definition.  Proceeding exactly as
before, we obtain induction claim \myeqref{eq:xistateevol1b} and the
claim \myeqref{eq:xistateevol3}.

The base case is simpler since for $t=0$, $\mfF_0$ is taken to be trivial. When $i\in\sC_N$ we have:
\begin{align*}
	\lim_{ N\to\infty} \E\left[\xi^1_{i\to j}\right] &= \lim_{N\to\infty}\E\biggl[\sum_{\ell \in \sC_N\bs j} A^0_{\ell i}f(1, 0) %
		+ \sum_{\ell\in[N]\bs \sC_N, j}A^0_{\ell i} f(1, 0) \biggr] \\
		&= \mu_1,	
\end{align*}
and for the variance:
\begin{align*}
	\lim_{ N\to\infty} \Var\left[\xi^1_{i\to j}\right] &= \lim_{N\to\infty}%
	\E\biggl[\sum_{\ell\in[N]\bs j}(\tilde A^0_{\ell i} f(1, 0))^2 \biggr]\\
	&= \lim_{N\to\infty} \frac{1}{N}\sum_{\ell\in[N]\bs\sC_N,j} (f(1, 0))^2\\
	&= \tau_{1}^2,
\end{align*}
by definition. It follows from the central limit theorem that $\xi^1_{i\to j}\convD%
\normal(\mu_1, \tau^2_{1})$ when $i\in\sC_N$. A very similar argument yields that
$\xi^1_{i\to j} \convD\normal(0, \tau^2_{1})$ when $i\notin \sC_N$. 
Eqs. \eqref{eq:xistateevol1a}, \eqref{eq:xistateevol1b},  \eqref{eq:xistateevol2} and \eqref{eq:xistateevol3} 
follow for $t=0$ using Lemma \ref{lem:mombnd}.

The proofs for $\xi^t_i$ follow from essentially the same argument except that the
required sums are modified to include the vertex $j$. Asymptotically in $N$, this 
has no effect on the result and we obtain the following limits:
\begin{align}
  \lim_{N\to\infty} \E[(\xi^t_{i})^m] &= \E[(\mu_t + Z_t)^m] \quad\text{ if } i\in \sC_N\label{eq:xistateevol6a}\\
  \lim_{N\to\infty} \E[(\xi^t_{i})^m] &= \E[(Z_t)^m] \quad\text{ otherwise}\label{eq:xistateevol6b}\\
	\lim_{N\to\infty} \frac{1}{|\sC_N|}\sum_{i\in \sC_N}(\xi^t_{i})^m &= \E[(\mu_t + Z_t)^m] \text{ in probability} \label{eq:xistateevol8}\\
	\lim_{N\to\infty} \frac{1}{N}\sum_{i\in[N]\backslash\sC_N} (\xi^t_{i})^m &= \E[(Z_t)^m] \text{ in probability.} \label{eq:xistateevol9}
\end{align}

Using Eqs.~\eqref{eq:xistateevol6a}, \eqref{eq:xistateevol6b} and Proposition \ref{prop:refresh} the result follows.
\end{proof}

	We can now prove Lemma \ref{lem:main}. For brevity, we show only \myeqref{eq:thetaconv2} 
	as the argument for \myeqref{eq:thetaconv1} is
	analogous. To show \myeqref{eq:thetaconv2} it suffices to show that, for any subsequence
	$\{N_k\}$ there exists a refinement $\{N_k'\}$ such that:
	\begin{align}
		\frac{1}{N_k'}\sum_{i=1\in[N_k']\bs\sC_{N_k'}} \psi(\theta^t_i) &= \E[\psi(Z_t)] \text{ a.s.} \label{eq:psisubseqconv}
	\end{align}

	Fix a subsequence $\{N_k\}$. By Chebyshev inequality, Lemma \ref{lem:sumvar} and Proposition
	\ref{prop:stateevol} there exists a refinement $\{N_k(1)\}\subseteq\{N_k\}$ such that:
	\begin{align*}
		\lim_{k\to\infty} \frac{1}{N_k(1)}\sum_{i\in [N_k(1)]\bs\sC_{N_k(1)}} \theta^t_i &= \E[Z_t] \text{ a.s.}
	\end{align*}
	By the same argument, for each $m \in\naturals$, there exists a refinement $\{N_k(m)\}\subseteq
	\{N_k(m-1)\}$ such that:
	\begin{align*}
		\lim_{k\to\infty} \frac{1}{N_k(m)}\sum_{i\in[N_k(m)]\bs\sC_{N_k(m)}} (\theta^t_i)^m &= \E[Z_t] \text{ a.s.}
	\end{align*}
	Let $N_k'$ be the sequence $N_k(k)$. Then, for all $m\ge 1$:
	\begin{align}
		\lim_{k\to\infty} \frac{1}{N'_k}\sum_{i\in[N_k']\bs\sC_{N_k'}} (\theta^t_i)^m &= \E[(Z_t)^m] \text{ a.s.} \label{eq:thmomconv}
	\end{align}
	We define the empirical measure $\mu_N(.)$ as follows:
	\begin{align*}
		\mu_N(\cdot) = \frac{1}{N}\sum_{i\in[N]\bs\sC_{N}} \delta_{\theta^t_i}(\cdot)
	\end{align*}
	\myeqref{eq:thmomconv} guarantees that, almost surely, the moments of $\mu_{N_k'}$
	converge to that of $Z_t$. By the moment method, \myeqref{eq:psisubseqconv} follows
	and we obtain the required result of \myeqref{eq:thetaconv2}.
%
%
\subsection{Proof of Lemma \ref{lem:approxstateevol}}

This is section is devoted to proving Lemma
\ref{lem:approxstateevol}. In particular, our derivation will justify
the construction of polynomials in the statement of the lemma, cf. Eq.~(\ref{eq:OptimalPolynomials}).

We will first consider the state evolution recursion (\ref{eq:stateevol1}),
(\ref{eq:stateevol2}) for a general sequence of functions
$\{f(\,\cdot\,,t)\}_{t\ge 0}$  (not necessarily polynomials). Since we
are only interested in the ratio $\mu_t/\tau_t$, there is no loss of
generality in assuming that $f$ is normalized in such a way that
$\tau_t^1=1$ for all $t$, i.e. $\E[f(Z,t)^2]=1$ for all $t$.
\begin{lemma}
Let $\mu_t$ be defined recursively for all $t\ge 0$ by letting
\begin{align}
  \label{eq:optstateevol}
  \mu_{t+1} = \lambda\kappa\, e^{\mu_t^2/2} \,,\;\;\;\;\;
  \mu_0 = 1\, .
\end{align}
 Further, given a sequence of functions $f\equiv
 \{f(\,\cdot\,,t)\}_{t\ge 0}$, such that $\E[f(Z,t)^2]=1$ for all $t$,
 let $\mu^{(f)}_t$ be the corresponding state evolution sequence
 defined by
\begin{align*}
  \mu^{(f)}_{t+1} = \lambda\kappa\E[f(\mu^{(f)}+Z)]  \,,\;\;\;\;\;
  \mu^{(f)}_0 = 1\,. 
\end{align*}
Then $\mu^{(f)}_t\le \mu_t$ for all $t$, with equality verifed for
$t>0$ if and
only if
  \begin{align*}
    f(z, \ell) = e^{\mu_\ell z - \mu_\ell^2} \text{ for }0 \le \ell
    \le t\, .
  \end{align*}
Further $\lim_{t\to\infty}\mu_t=\infty$ if and only if $\lambda\kappa>e^{-1/2}$.
	\label{lem:optstateevol}
\end{lemma}
\begin{proof}
For the initial condition $\mu^{(f)}_0 = 1, \tau_0 = 0$, it is easy
to see that the choice of normalization ensures that we need
only fix $f(1, 0) = 1$ which is satisfied by the choice above. We have, for 
$Z\sim\normal(0, 1)$, and $\ell\ge 0$:
  \begin{align*}
    \mu^{(f)}_{\ell+1} &= \lambda\kappa\,\E[f(\mu^{(f)}_\ell+Z)] \\
    &= \lambda\kappa\,\int_{\reals} f(z) e^{-(z - \mu^{(f)}_\ell)^2/2} \frac{\mathrm{d}z}{\sqrt{2\pi}} \\
    &= \lambda\kappa\, e^{-(\mu^{(f)}_\ell)^2/2}\E\left[f(Z)e^{\mu^{(f)}_\ell Z}\right] \\
    &\le \lambda\kappa\,e^{-(\mu^{(f)}_\ell)^2/2}\left(\E\left[ (f(Z, \ell)^2 \right]\right)^{1/2} e^{(\mu^{(f)}_\ell)^2},
\end{align*}
where the inequality follows from Cauchy-Schwartz. By our choice
of normalization we obtain:
\begin{align*}
  \mu^{(f)}_{\ell+1} &\le \lambda\kappa\, e^{(\mu^{(f)}_\ell)^2/2}.
\end{align*}
Since the inequality is satisfied as equality only  by the choice
$f(z, \ell) = e^{\mu_\ell z - \mu_\ell^2}$, we have proved 
that $\mu_t^{(f)} = \mu_t$ only for this choice.

The last statement (namely $\mu_t\to\infty$ if and only if
$\lambda\kappa>e^{-1/2}$) is a simple calculus exercise.
\end{proof}
%
%
  
  We are now in position to prove Lemma \ref{lem:approxstateevol}.
\begin{proof}[Proof of Lemma  \ref{lem:approxstateevol}]
  Let $\{\mu_t\}_{t\ge 0}$ be given as per Eq.~(\ref{eq:optstateevol})
  and define $t^* \equiv \inf\{t: \mu_{t^*} > 2M\}$. The condition $\lambda\kappa>e^{-1/2}$ 
  ensures that $t^*$ is finite. For a fixed $d$, 
define the mappings $g,\hat g_d:\reals\times\reals\to\reals$ by letting $g(z, \mu) = e^{\mu z}$ 
  and $\hat g(z, \mu) = \sum_{k = 0}^d \mu^k z^k/k!$. Then, since the
  Taylor series of the exponential has infinite radius of convergence,
  we have, for all $z$, $\mu\in\reals$,
  \begin{align}
    \lim_{d\to\infty}\hat g_d(z, \mu) = g(z, \mu)\, \label{eq:Pointwise}.
  \end{align}
In the rest of this proof we will --for the sake of simplicity-- omit the subscript $d$.
 
For any  $\mu\in\reals$ and $Z\sim\normal(0,1)$, we define:
  \begin{align*}
    G(\mu) = \frac{1}{L}\E\left[\lambda \kappa \, g(\mu + Z, \mu)\right]\\
    \hat G(\mu) = \frac{1}{\hat L}\E\left[\lambda\kappa\,\hat g(\mu + Z, \mu)\right],
  \end{align*}
  where $L = (\E[g(Z, \mu)^2])^{1/2}$ and $\hat L = (\E[\hat g(Z, \mu)^2])^{1/2}$.
  We first obtain that:
  \begin{align}
    |G(\mu) - \hat G(\mu)| &\le G(\mu)\frac{|L - \hat L|}{\hat L} + %
    \frac{\lambda\kappa}{\hat L}\left\lvert \E\left[(g(\mu+Z, \mu)- \hat g(\mu+Z, \mu))e^{\mu Z}\right]\right\rvert.
  \end{align}
Note that $|g(z,\mu)|,|\hat g(z,\mu)|\le e^{\mu |z|}$. It follows
from Eq.~(\ref{eq:Pointwise}) and dominated convergence that $\hat
L\to L$ and $\E[\hat g(\mu+Z, \mu)]\to \E[g(\mu+Z, \mu)]$ as
$d\to\infty$.

  By compactness, for any $\delta>0$, we can choose $d^* < \infty$ such that
  $|G(\mu) - \hat G(\mu)| \le \delta$ for $0 \le \mu \le 2M$. Here $d^*$
  is a function of $\delta, M$. Note that we can now rewrite the state
  evolution recursions as follows:
  \begin{align*}
    \mu_{\ell+1} &= G(\mu_\ell),\\
    \hat \mu_{\ell+1} &= \hat G(\hat\mu_\ell),
  \end{align*}
  with $\mu_0 = \hat\mu_0 = 1$. Define $\Delta_\ell = |\mu_\ell - \hat\mu_\ell|$. 
  Then using the fact that $G(\mu)$ is convex and that $G'(\mu) = \lambda\kappa\mu\, e^{\mu^2/2}$
  is bounded by $M' = G'(2M)$ we obtain:
  \begin{align*}
    \hat \mu_{\ell+1} &= \hat G(\hat \mu_\ell) \\
    &\ge G(\hat \mu_\ell) - \delta \\
    &\ge G(\mu_\ell) - M'|\mu_\ell - \hat\mu_\ell| - \delta\\
    &= \mu_{\ell+1} - M' \Delta_\ell - \delta.
  \end{align*}
  This implies:
  \begin{align*}
    \Delta_{\ell+1} &\le M'\Delta_\ell + \delta.
  \end{align*}
  By induction, since $\Delta_0 = |\mu_0 - \hat\mu_0| = 0$,
  we obtain:
  \begin{align*}
    \Delta_\ell &\le \left(\sum_{k = 0}^{\ell-1} (M')^k \right) \delta \\
    &= \frac{M'^\ell - 1}{M'-1}\delta.
  \end{align*}
  Now, choosing $d^*$ such that $\delta = M(M'-1)/2(M'^{t^*} - 1)$ we obtain that
  $\Delta_{t^*} \le M/2$, implying that $\hat\mu_{t^*} > 3M/2 > M$.
\end{proof}
%
%

  \subsection{Proof of Lemma \ref{pro:algproof}}
	Let $A|_{\tsC_N}$ be the matrix $A$, restricted to the rows (and columns) 
	in $\tsC_N$. 
	Also let $v \in \reals^{|\tsC_N|}$ denote the unit norm indicator vector on 
	$\tsC_N\cap\sC_N$, i.e.
	\begin{align*}
	  v_i &= \begin{cases} |\tsC_N\cap\sC_N|^{-1/2} &\text{ if } i
            \in \tsC_N\cap\sC_N \, ,\\
	    0 &\text{ otherwise}. 
	  \end{cases}
	\end{align*}  
	Define $\tilde A|_{\tsC_N}$ to be a centered
	matrix such that:
	\begin{align*}
	  A|_{\tsC_N} &= \frac{\lambda |\tsC_N\cap\sC_N|}{\sqrt N} vv^\sT + \tilde A|_{\tsC_N} 
	\end{align*}
	Throughout this proof we assume for simplicity that $\kappa_N
        = \kappa$, i.e. $|\sC_N| = \kappa\sqrt{N}$  for some constant
        $\kappa$ independent of $N$. The case of $\kappa_N$ dependent
        on $N$ with $\lim_{N\to\infty}\kappa_N=\kappa$  can be covered by
	a vanishing shift in the constants presented.
	
	Assume that $\tsC_N$ is a fixed subset selected independently
        of $A$. Then the matrix
	$\tilde A|_{\tsC_N}$ has independent, zero-mean entries which are subgaussian
	with scale factor $\rho/N$. Let $u$ denote the principal
        eigenvector of $A|_{\tsC_N}$. The set $\sB_N\subset[N]$ consists of the indices 
	of the  $|\sC_N|$ entries of $u$ with largest absolute value.  

	We first show that the set $\sB_N$ contains a large fraction of $\sC_N$. By the condition
	on $\tsC_N$, we have that $\norm{A|_{\tsC_N} - \tilde A|_{\tsC_N}}_2\ge\lambda\kappa(1-\eps)$.
        By Lemma \ref{lem:2normbnd} in Appendix \ref{app:Tools}, 
	 for a fixed $\delta$,  $\norm{\tilde A|_{\tsC_N}}_2 \le \lambda(1-\eps)\kappa\delta$ with 
	probability at least $2(5\xi)^Ne^{-N(\xi-1)}$ where
	$\xi = \delta^2/32\rho\eps$.

	Using matrix perturbation theory, we get
	\begin{align*}
	  \norm{u - v}_2 &\le \sqrt{2}\sin \theta(u, v) \\
		&\le \sqrt{2}\frac{\norm{\tilde A|_{\tsC_N}}_2}%
		{\lambda(1-\eps)\kappa_N - \norm{\tilde A|_{\tsC_N}}_2} \\
		&\le 1.9\, \delta,
	\end{align*}
	where the second inequality follows by the \textsc{sin}
        $\theta$ theorem \cite{davis1970sin}.
 
	We run $t_{**} = O(\log N/\delta)$ iterations of the power
        method, with  initialization $u^0= (1,1,\dots,1)^{\sT}/|\tsC_N|^{1/2}$.
	By the same perturbation argument, there is a $\Theta(1)$ gap between the largest and
        second largest eigenvalue of $A|_{\tsC_N}$, and $\<u^0,u\> \ge
        N^{-c}$. It follows by a standard argument that the output
        $u^{**}$ of the power method is  an approximation
	to the leading eigenvector $u$ with a fixed error $\|u - u^{**}\| \le \delta/10$.
This
	implies that $\|u^{**} - v\| \le 2\delta$, by the triangle inequality.
	Let $u^\perp$ ($u^\parallel$)
	denote the projection of $u^{**}$ orthogonal to (resp. onto) $v$. 
	Thus we have $\norm{u^\perp}_2^2 \le 4\delta^2$.
	It follows that at most $36\delta^2|\tsC_N\cap\sC_N|$ entries in $u^\perp$ have magnitude
	exceeding $(1/3)|\tsC_N\cup\sC_N|^{-1/2}$.
	Notice that $u^\parallel = u^{**} - u^\perp$ and $u^\parallel$ is a multiple of $v$. Consequently,
	we can 
	assume $\sB_N$ is selected using $u^\parallel$, instead of $v$. This observation along with
	the bound above guarantees that at most $36\delta^2|\tsC_N\cap\sC_N|$ entries are
	misclassified, i.e.
	\begin{align}
	  |\sB_N\cap\sC_N| &\ge (1-36\delta^2)|\tsC_N\cap\sC_N| \\
	  &\ge (1-\delta)(1-\eps)|\sC_N|. \label{eq:BcapCN}
	\end{align}
	Here we assume $\delta\le 1/36$. 

	The above argument proves that the desired result for any fixed
        set $\tsC_N$ independent of $A$ with a probability at
	least $1 - 2(5\xi)^Ne^{-N\xi/2}$ where $\xi= \delta^2/32\rho\eps$,
	for a universal constant $c$ and $N$ large enough. In order to
        extend it to all sets $\tsC_N$ (possibly dependent on the
        matrix $A$), we can take a union bound over all possible choices of $\tsC_N$
	and obtain the required result.
	For all  $N$ large enough, the number of choices satisfying the  conditions
	of Lemma \ref{pro:algproof}
	is bounded by:
	\begin{align*}
	  \# N(\eps) &\le 2e^{N(\eps - \eps\log\eps)}. 
	\end{align*}
	Choosing $\delta = \eps^{1/4}$, it follows from the union bound that
	for some $\eps$ small enough, we have that \myeqref{eq:BcapCN} holds with 
	probability at least $1 - 4e^{-N\upsilon'}$ where $\upsilon'(\rho, \eps)\to \infty$
	as $\eps\to 0$.
	Recall that the score $\zeta^{\sB_N}_{\bar\rho}(i)$ for a vertex $i$ is given by:
	\begin{align*}
	  \zeta^{\sB_N}_{\bar\rho}(i) &= \frac{1}{|\sC_N|}\sum_{j\in \sB_N} W_{ij}\ind_{\{|W_{ij}|\le \bar\rho\}}\\
	  &= \frac{1}{|\sC_N|}\sum_{j\in\sC_N}W'_{ij} + \frac{1}{|\sC_N|}\left( \sum_{j\in \sB_N\bs\sC_N}W'_{ij} - \sum_{j\in \sC_N\bs \sB_N} W'_{ij}\right),
	\end{align*}
	where $W'_{ij} = W_{ij}\ind_{\{\lvert W_{ij}\rvert \le \bar\rho\}}$. 
	The truncated variables are subgaussian with the parameters $\lambda'_\ell, \rho'$ for
	$\ell = 0, 1$ as according $W_{ij} \sim Q_0, Q_1$. (Here $\lambda'_\ell$ denote the
	means after truncation)
	Also, for $\bar\rho$ large enough, 
	we may take 
	\begin{align*}
	  \lambda_1' &\ge \frac{7}{8}\lambda \, ,\\
	  \lambda_0' &\le \frac{1}{8}\lambda\, , \\
	  \rho' &\le  2\rho. 
	\end{align*}
	It follows that since the sum $\left(\sum_{j\in\sC_N}W'_{ij}\right)/\lvert\sC_N\rvert$ 
	is subgaussian with parameters $\lambda'_\ell, \rho/|\sC_N|$ the following
	holds with high probability:
	\begin{align*}
	  \zeta^{\sB_N}_{\bar\rho}(i) 
	    \ge \frac{3}{4}\lambda - 2\bar\rho(\delta+\eps) &\text{ if } i \in \sC_N \\
	     \zeta^{\sB_N}_{\bar\rho}(i)  \le \frac{1}{4}\lambda + 2\bar\rho(\delta+\eps) &\text{ otherwise.}
	\end{align*}
	Choosing $\eps \le (\lambda/20\bar\rho)^4$ yields the desired result.

\section{The sparse graph case: Algorithm and proof of Theorem
  \ref{thm:Sparse}}
  \label{sec:Sparse}

  \def\tTree{\widetilde{\sf T}{\sf ree}}
\def\root{\circ}

In this section we consider the general hidden set problem
on locally tree-like graphs, as defined in the introduction. 
We will introduce the reconstruction algorithm and the basic idea of
its analysis. A formal proof of Theorem \ref{thm:Sparse} 
will be presented in Section \ref{sec:TreeLemmas} and builds on these ideas.

Throughout this section  we consider a sequence of locally tree-like
graphs $\{G_N\}_{N\ge 1}$, $G_N = ([N], E_N)$, indexed by the number
of vertices $N$. For notational simplicity, we shall assume that these
graphs are $(\Delta+1)$-regular, although most of the ideas can be
easily generalized. We shall further associate to each vertex $i$ a
binary variable $X_i$, with $X_i=1$ if $i\in \sC_N$ and $X_i=0$
otherwise.
We write $X=(X_i)_{i\in[N]}$ for the vector of these variables.
It is mathematically convenient to work with a slightly different
model for the vertex labels $X_i$: we will assume that the $X_i$ are
i.i.d. such that:
\begin{align*}
  \P(X_i=1) &= \frac{\kappa}{\sqrt{\Delta}}\left( 1+\frac{\kappa}{\sqrt{\Delta}} \right)^{-1}. 
\end{align*}
For convenience of exposition, we also define:
\begin{align*}
  \tkappa(\Delta) &\equiv \kappa\left( 1+\frac{\kappa}{\sqrt{\Delta}} \right)^{-1}.
\end{align*}

Notice that this leads to a set 
$\sC_N = \{i\in [N]: X_i = 1\}$ that has a random size
which concentrates sharply around $N\tkappa/\sqrt{\Delta}$. This is
a slightly different model from what we consider earlier: $\sC_N$ is
uniformly random and of a fixed size. However, if we condition
on the size $|\sC_N|$, the i.i.d. model reduces to the earlier model. 
We prove in Appendix \ref{sec:iidProof}, 
that the results of the i.i.d. model still hold for the earlier
model. In view of this, throughout this section we will stick
to the i.i.d model.

In order to motivate the algorithm, consider the conditional
distribution of $W$ given $X$, and assume for notational simplicity
that $Q_0$, $Q_1$ are discrete distributions. We then have
$\prob(W|X=x) = \prod_{(i,j)\in E_N} Q_{x_ix_j}(W_{ij})$. Here the subscript
$x_ix_j$ means the product of $x_i$ and $x_j$. The posterior
distribution of $x$ is therefore a Markov random field (pairwise
graphical model) on $G_N$:
\begin{align*}
\prob(X=x|W) = \frac{1}{Z(W)}\, \prod_{(i,j)\in E_N}
Q_{x_ix_j}(W_{ij})\prod_{i\in[N]}
\Big(\frac{\tkappa}{\sqrt{\Delta}}\Big)^{x_i}\Big(1-\frac{\tkappa}{\sqrt{\Delta}}\Big)^{1-x_i}\, .
\end{align*}
Here $Z(W)$ is an appropriate normalization. \emph{Belief propagation} (BP) is 
a heuristic method for estimating the marginal
distribution of this posterior, see \cite{wainwright2008graphical,MezardMontanari,koller2009probabilistic} 
for introductions from several points of view. 
For the sake of simplicity, we shall describe the algorithm for the case $Q_1=
\delta_{+1}$, $Q_0= (1/2)\delta_{+1}+(1/2)\delta_{-1}$, whence $W_{ij} \in\{+1,-1\}$.
At each iteration $t$, the algorithm updates `messages'
$\gamma^t_{i\to j}$,  $\gamma^t_{j\to i}\in\reals_+$,  for each
$(i,j)\in E_N$. As formally clarified below,
  these messages correspond to `odds ratios' for vertex $i$ to be in
  the hidden set.

Starting from $\gamma^0_{i\to j} =1$ for all $i,j$,
messages are updated as follows:
\begin{align} \label{eq:BPeqnSparse1}
\gamma^{t+1}_{i\to j} &= \kappa\,\prod_{\ell\in\di\setminus j}
\left(\frac{1+(1+W_{i,\ell})\gamma^t_{\ell\to
      i}/\sqrt{\Delta}}{1+\gamma^t_{\ell\to i}/\sqrt{\Delta}}\right)\,
.
\end{align}
where $\di$ denotes the set of neighbors of $i$ in $G_N$. We
further compute the vertex quantities $\gamma^t_i$ as 
\begin{align} \label{eq:BPeqnSparse2}
\gamma^{t+1}_{i} &= \tkappa\,\prod_{\ell\in\di}
\left(\frac{1+(1+W_{i,\ell})\gamma^t_{\ell\to
      i}/\sqrt{\Delta}}{1+\gamma^t_{\ell\to i}/\sqrt{\Delta}}\right)\,
.
\end{align}
Note that $\gamma_i^t$ is a function of the (labeled) neighborhood
$\Ball_{G_N}(i;t)$. The nature of this function is clarified by the next
result, that is an example of a standard result in the literature on
belief propagation
\cite{wainwright2008graphical,MezardMontanari,koller2009probabilistic}.
\begin{proposition}\label{propo:BP_Exact_Tree}
Let $W_{\Ball_{G_N}(i;t)}$ be the set of edge labels in the subgraph $\Ball_{G_N}(i;t)$.
If $\Ball_{G_N}(i;t)$ is a tree, then 
\begin{align*}
\frac{\prob(X_i=1|W_{\Ball_{G_N}(i;t)})
}{\prob(X_i=0|W_{\Ball_{G_N}(i;t)}) } =
\frac{\gamma_i^t}{\sqrt{\Delta}}\, .
\end{align*}
\end{proposition}

Given this result, we can attempt to estimate $\sC_N$ on locally tree-like
graphs by running BP for $t$ iterations and subsequently thresholding
the resulting odds-ratios. In other words we let
\begin{align}
\hsC_N \equiv\big\{i\in [N]:\; \gamma^t_i\ge \sqrt{\Delta}\big\}\, . \label{eq:BPSet}
\end{align}
By Proposition \ref{propo:BP_Exact_Tree}, this corresponds to
maximizing the posterior probability
$\prob(X_i=x_i|W_{\Ball_{G_N}(i;t)})$ for all vertices $i$ such that
$\Ball_{G_N}(i;t)$ is a tree. This in turn minimizes the
misclassification rate $\prob(i\in \sC_N;i\not\in\hsC_N)+\prob(i\in
\hsC_N;i\not\in\sC_N)$. The resulting error rate is
\begin{align}
\Big(1-\frac{\tkappa}{\sqrt{\Delta}}\Big) \prob\big(\gamma^t_i\ge \sqrt{\Delta}\big|X_i=0\big)
+\frac{\tkappa}{\sqrt{\Delta}}\prob\big(\gamma^t_i<\sqrt{\Delta}\big|X_i=1\big)\, .
\end{align}

In order to characterize this misclassification rate, we let
 $\tTree(t)$ denote the regular $t$-generations
with degree  $(\Delta+1)$ at each vertex except the leaves, rooted at vertex $\root$ and labeled as follows.
Each vertex $i$ is labeled with $X_i\in\{0,1\}$ independently 
with $\prob(X_i=1)=\tkappa/\sqrt{\Delta}$. 
Each edge $(i,j)$ has label an independent  $W_{ij}\sim
Q_{1}$ if $X_i=X_j=1$ and    $W_{ij}\sim
Q_{0}$ otherwise. 

Let $\tgamma^t(x_\root)$ a random variable distributed as the  odds ratio for $X_\root=1$ on $\tTree(t)$
when the true root value is $x_\root$
\begin{align}
\tgamma^t(x_\root) \equiv \sqrt{\Delta}\,
\frac{\prob(X_\root=1|W_{\tTree(t)})}{\prob(X_\root=0|W_{\tTree(t)})}\,
,\;\;\;\;\; W_{\tTree(t)}\sim\prob(W_{\tTree(t)}=\,\cdot\,|X_\root=x_\root)\, . \label{eq:TgammaDef}
\end{align}
The following characterization is a direct consequence of 
the fact Proposition \ref{propo:BP_Exact_Tree} and the fact
that $G_N$ is locally tree-like. For completeness, we 
provide a proof in Appendix \ref{sec:TreeErrorProof}
\begin{proposition}\label{propo:TreeError}
Let $\hsC_N$ be the estimated hidden set for the BP rule
(\ref{eq:BPSet}) after $t$ iterations. We then have
\begin{align*}
  \lim_{N\to\infty} \frac{1}{N}\,\E[|\sC_N\triangle \hsC_N|]= \Big(1-\frac{\tkappa}{\sqrt{\Delta}}\Big) \prob\big(\tgamma^t(0)\ge\sqrt{\Delta}\big)
+\frac{\tkappa}{\sqrt{\Delta}}\prob\big(\tgamma^t(1)<\sqrt{\Delta}\big)\, .
\end{align*}
Further, if $\hsC_N$ is estimated by \emph{any} $t$-local algorithm, then
$\lim\inf_{N\to\infty} N^{-1}\,\E[|\sC_N\triangle \hsC_N|]$
is at least as large as the right-hand side.
\end{proposition}

We have therefore reduced the proof of Theorem \ref{thm:Sparse} to
controlling the distribution of the random variables $\tgamma^t(0)$,
$\tgamma^t(1)$.
These can be characterized by a recursion over $t$.
For $\kappa$ small we have the following.
\begin{lemma}\label{lemma:TreeSmallKappa}
Assume $\kappa <1/\sqrt{e}$. Then there exists constants
$\gamma_*<\infty$, $\delta_* = \delta_*(\kappa)$ and 
$\Delta_*=\Delta_*(\kappa)<\infty$ such that,
for all $\Delta>\Delta_*(\kappa)$ and all $t\ge 0$, we have
\begin{align*}
%
\prob(\tgamma^t(1)\le 5\gamma_*)\ge \frac{3}{4}\,.
\end{align*}
\end{lemma}
For large $\kappa$, we have instead the following.
\begin{lemma}\label{lemma:TreeLargeKappa}
Assume $\kappa >1/\sqrt{e}$. Then there exists $c_*=c_*(\kappa)>0$
 $\Delta_*=\Delta_*(\kappa)<\infty$, $t_*=t_*(\kappa,\Delta)<\infty$ such that,
for all $\Delta>\Delta_*(\kappa)$  we have
\begin{align*}
\Big(1-\frac{\tkappa}{\sqrt{\Delta}}\Big) \prob\big(\tgamma^t(0)\ge\sqrt{\Delta}\big)
+\frac{\tkappa}{\sqrt{\Delta}}\prob\big(\tgamma^t(1)<\sqrt{\Delta}\big)\le e^{-c_*\sqrt{\Delta}}\,.
\end{align*}
\end{lemma}

  Lemma \ref{lemma:TreeSmallKappa} and Proposition \ref{propo:TreeError}
  together imply one part of Theorem \ref{thm:Sparse}. Indeed, for $\Delta$ 
  large enough, we have that the misclassification error is $\Omega(N/\sqrt{\Delta})$, 
  which is the same order as choosing a random subset of size $\tkappa N/\sqrt{\Delta}$. 
  Similarly, Lemma \ref{lemma:TreeLargeKappa} in conjunction with Proposition
  \ref{propo:TreeError} yields the second half of Theorem \ref{thm:Sparse}. 

%
%
\section{Proof of Lemma \ref{lemma:TreeSmallKappa} and
\ref{lemma:TreeLargeKappa}}
\label{sec:TreeLemmas}

In this section we  prove Lemma \ref{lemma:TreeSmallKappa} and
\ref{lemma:TreeLargeKappa} that are the key technical results
leading to Theorem \ref{thm:Sparse}. We start by establishing some
facts that are useful in both cases and then pass to the proofs of the
two lemmas.
%
%
\subsection{Setup: Recursive construction of $\tgamma^t(0)$, $\tgamma^t(1)$}

As per Proposition \ref{propo:BP_Exact_Tree}, the likelihood ratio
$\tgamma^t(x_\root)$ can be computed by applying the BP recursion on the
tree $\tTree(t)$.
In order to set up this recursion, let 
 $\Tree(t)$ denote the $t$-generation tree, with root
 of degree $\Delta$, and other non-leaf vertices of
 degree $\Delta+1$. The tree $\Tree(t)$ carries labels
$x_i$, $W_{ij}$ in the same fashion as $\tTree(t)$. Thus, 
$\Tree(t)$ differs from $\tTree(t)$ only in the root degree.
We then let
\begin{align}
\gamma^t(x_\root) \equiv \sqrt{\Delta}\,
\frac{\prob(X_\root=1|W_{\Tree(t)})}{\prob(X_\root=0|W_{\Tree(t)})}\,
,\;\;\;\;\; W_{\Tree(t)}\sim\prob(W_{\Tree(t)}=\,\cdot\,|X_\root=x_\root)\, . \label{eq:GammaDef}
\end{align}
It is then easy to obtain the distributional recursion
(here and below $\ed$ indicates equality in distribution)
\begin{align}
\gamma^{t+1}(0) &\ed \kappa\,\prod_{\ell=1}^\Delta
\left(\frac{1+(1+A^t_\ell)\gamma^t_\ell(x_\ell)/\sqrt{\Delta}}{1+\gamma^t_\ell(x_\ell)/\sqrt{\Delta}}\right)\,
,\label{eq:FirstTreeRecursion}\\
\gamma^{t+1}(1) &\ed  \kappa\,\prod_{\ell=1}^\Delta
\left(\frac{1+(1+\tA^t_\ell)\gamma^t_\ell(x_\ell)/\sqrt{\Delta}}{1+\gamma^t_\ell(x_\ell)/\sqrt{\Delta}}\right)\,
. \label{eq:SecondTreeRecursion}
\end{align}
This recursion is initialized with $\gamma^0(0) = \gamma^0(1) = \kappa$.
Here $\gamma^t_\ell(0),\gamma_\ell^t(1)$, $\ell\in [\Delta]$ are
$\Delta$ i.i.d. copies of $\gamma_t(0), \gamma_t(1)$, $A^t_i$, $i\in [\Delta]$,  
are i.i.d. uniform in $\{\pm 1\}$, $x_i$, $i\in[\Delta]$ are i.i.d. Bernoulli with 
$\P(x_i =1)=\tkappa/\sqrt{\Delta}$. Finally $\tA^t_{i}=
A^t_i$ if $x_i=0$ and $\tA^t_{i} = 1$ if $x_i=1$.

The distribution of  $\tgamma^t(0)$, $\tgamma^t(1)$ can then be
obtained  from the one of $\gamma^t(0)$, $\gamma^t(1)$ as follows:
\begin{align}
  \tgamma^{t+1}(0) &\ed \kappa\,\prod_{\ell=1}^{\Delta+1}%
  \left(\frac{1+(1+A^t_\ell)\gamma^t_\ell(x_\ell)/\sqrt{\Delta}}{1+\gamma^t_\ell(x_\ell)/\sqrt{\Delta}}\right)\, ,\label{eq:FirstVertexRecursion}\\
  \tgamma^{t+1}(1) &\ed \kappa\,\prod_{\ell=1}^{\Delta+1}%
  \left(\frac{1+(1+\tA^t_\ell)\gamma^t_\ell(x_\ell)/\sqrt{\Delta}}{1+\gamma^t_\ell(x_\ell)/\sqrt{\Delta}}\right)\,. \label{eq:SecondVertexRecursion}
\end{align}

%
%
\subsection{Useful estimates}

A first useful fact is the following relation between the moments of
$\gamma^t(0)$ and $\gamma^t(1)$.
\begin{lemma}
  Let $\gamma(0), \gamma(1), \tgamma(0), \tgamma(1)$ be defined as in
  Eqs.~\eqref{eq:FirstTreeRecursion}, \eqref{eq:SecondTreeRecursion},
  \eqref{eq:FirstVertexRecursion} and \eqref{eq:SecondVertexRecursion}.
  Then, for each positive integer $a$ we have:
  \begin{align*}
    \E\left[(\gamma^t(0))^a\right] &= \kappa\,\E\left[ (\gamma^t(1))^{a-1} \right] \\
    \E\left[ (\tgamma^t(0))^a \right] &= \kappa\,\E\left[ (\tgamma^t(1))^{a-1} \right].
  \end{align*}\label{lem:gammaMoments}
\end{lemma}
\begin{proof}
It suffices to show that:
  \begin{align*}
    \frac{\md P_{\gamma^t(1)}}{\md P_{\gamma^t(0)}}(\gamma) &=
    \frac{\gamma^t}{\kappa}\, ,
  \end{align*}
  where the left-hand side denotes the Radon-Nikodym derivative of
  $P_{\gamma^t(1)}$ with respect to  $P_{\gamma^t(0)}$. Let $\nu^t$
  denote the posterior probability of $x_\root = 0$ given the labels on
  $\Tree(t)$.
  Let $\nu^t(x_{\root})$ be distributed as  $\nu^t$ conditioned on the
  event $X_{\root} = x_{\root}$ for $x_{\root} = 0, 1$. In other words
\begin{align*}
\nu^t &\equiv 
\prob(X_\root=1|W_{\Tree(t)})\, ,\\
\nu^t(x_\root) &\equiv 
\prob(X_\root=1|W_{\Tree(t)})\,
,\;\;\;\;\; W_{\Tree(t)}\sim\prob(W_{\Tree(t)}=\,\cdot\,|X_\root=x_\root)\, .
\end{align*}
  By Bayes rule we then have:
  \begin{align*}
    \frac{\md P_{\nu^t(0)}}{\md P_{\nu^t}} &= \frac{\nu^t}{1 -
      \tkappa/\sqrt\Delta} \, ,\\
    \frac{\md P_{\nu^t(1)}}{\md P_{\nu^t}} &=
    \frac{1-\nu^t}{\tkappa/\sqrt\Delta}\, .
  \end{align*}
  Using this and the fact that $\nu^t = (1+\gamma^t/\sqrt\Delta)^{-1}$ by
Eq.~(\ref{eq:GammaDef}), we get 
  \begin{align*}
    \frac{\md P_{\nu^t(0)}}{\md P_{\nu^t}} &= \frac{1}{\left(
        1+\gamma^t/\sqrt\Delta \right)( 1- \kappa/\sqrt\Delta )} \, ,\\
    \frac{\md P_{\nu^t(1)}}{\md P_{\nu^t}} &=
    \frac{\gamma^t/\sqrt\Delta}{\left( 1+\gamma^t/\sqrt\Delta
      \right)\tkappa/\sqrt\Delta}\, .
  \end{align*}
  It follows from this and that the mapping from $\nu^t$ to the likelihood
  $\gamma$ is bijective and Borel that:
  \begin{align*}
    \frac{\md P_{\gamma^t(0)}}{\md P_{\gamma^t(1)}} &= \gamma^t\left( \frac{\tkappa}{1-\tkappa/\sqrt{\Delta}} \right)^{-1}\\
    &= \frac{\gamma^t}{\kappa}.
  \end{align*}
  Here the last equality follows from the definition of $\tkappa$. A similar
  argument yields the same result for $\tgamma^t(0)$ and $\tgamma(1)$.
\end{proof}

Our next result is a general recursive upper bound on the moments of $\gamma^t(1)$.
\begin{lemma}
  \label{lem:gammaRawMoments}
  Consider random variables $\gamma^t(0), \gamma^t(1), \tgamma^t(0),
  \tgamma^t(1)$ that satisfy the distributional recursions in 
  Eqs.~\eqref{eq:FirstTreeRecursion}, \eqref{eq:SecondTreeRecursion},
  \eqref{eq:FirstVertexRecursion} and \eqref{eq:SecondVertexRecursion}.
  Then we have that, for each $t \ge 0$:
  \begin{align*}
    \E\left[ \gamma^{t+1}(1) \right] &\le \kappa\exp\left(
      \kappa\,\E[\gamma^t(1)] \right)\, ,\\
    \E\left[ \gamma^{t+1}(1)^2 \right]
    &\le\kappa^2\,\exp\left(3\kappa\,\E[\gamma^t(1)] \right)\,  ,\\
    \E\left[ \gamma^{t+1}(1)^3 \right] &\le \kappa^3\,\exp\left(
      10\kappa\,\E[\gamma^t(1)] \right)\, .
  \end{align*}
  Moreover, we also have:
  \begin{align*}
    \E\left[ \tgamma^{t+1}(1) \right] &\le \kappa\exp\left(
      \kappa\,\E[\gamma^t(1)] \right)\left( 1 +
      \frac{\E[\gamma^t(1)}{\Delta} \right)\, ,\\
      \E\left[ \tgamma^{t+1}(1)^2 \right]
      &\le\kappa^2\,\exp\left(3\kappa\,\E[\gamma^t(1)]\right)\left(
        1 + \frac{3\E[\gamma^t(1)}{\Delta} \right)\, ,\\
    \E\left[ \tgamma^{t+1}(1)^3 \right] &\le \kappa^3\,\exp\left(
      10\kappa\,\E[\gamma^t(1)]\right) \left( 1 +
      \frac{10\E[\gamma^t(1)}{\Delta}\right)\, .
  \end{align*}
\end{lemma}
\begin{proof}
  Consider the first moment $\E[\gamma^{t+1}(1)]$. By taking
  expectation of Eq.~(\ref{eq:SecondTreeRecursion}) over
  $\{A_\ell,x_\ell\}_{1\le \ell\le \Delta}$, we get
  \begin{align*}
    \E\left[\gamma^{t+1}(1)|\{\gamma^t_\ell\}_{1\le \ell\le \Delta}\right] & =\kappa\,\prod_{\ell=1}^\Delta
    \left[\Big(1-\frac{\tkappa}{\sqrt{\Delta}}\Big)+\frac{\tkappa}{\sqrt{\Delta}}\frac{1+2\gamma^t_\ell(1)/\sqrt{\Delta}}{1+\gamma^t_\ell(1)/\sqrt{\Delta}}\right]\\
    &= \kappa\,\prod_{\ell=1}^\Delta
    \left[ 1+\frac{\tkappa}{\Delta} \frac{\gamma^t_\ell(1)}{1+\gamma^t_\ell(1)/\sqrt{\Delta}}\right]\\
    & \le \kappa\,\prod_{\ell=1}^\Delta
    \left[1+\frac{\kappa}{\Delta}\; \gamma^t_\ell(1)\right]
  \, .
  \end{align*}
  The last inequality uses the non-negativity of $\gamma^t(1)$
  and that $\kappa > \tkappa$. 
  Taking expectation over $\{\gamma^t_\ell\}_{1\le \ell\le \Delta}$, and using
  the inequality $(1+x)\le e^x$, we get 
  \begin{align*}
  \E[\gamma^{t+1}(1)]\le \kappa\exp\left(\kappa\,
  \E[\gamma^{t}(1)]\right)\, ,
  \end{align*}
  The claim for $\E[\tgamma^{t+1}(1)]$ follows from the same argument, 
  except we include $\Delta+1$ factors above and retain only the last.
  
  Next take the second moment $\E[\gamma^{t+1}(1)]$. Using
  Eq.~(\ref{eq:SecondTreeRecursion}) and proceeding as above 
  we get:
  
  \begin{align*}
    \begin{aligned}
    \E[\gamma^{t+1}(1)^2] &= \kappa^2\, \E\left[ \prod_{\ell= 1}^{\Delta}\bigg( 1 + %
    \frac{\tA_\ell \gamma^t_\ell(x_\ell)/\sqrt\Delta}{1 +\gamma^t_\ell(x_\ell)/\sqrt\Delta} \bigg)^2\right]\\ 
    &=\kappa^2\, \bigg[1+\Big(1-\frac{\tkappa}{\sqrt\Delta}\Big)\frac{1}{\Delta}
      \E\left[\Big(\frac{\gamma^t(0)^2}{1+\gamma^t(0)/\sqrt{\Delta}}\Big)^2\right] \\
      &\qquad+\frac{\tkappa}{\sqrt\Delta}\E\left(\frac{2\gamma^t(1)\sqrt\Delta+3\gamma^t(1)^2/\Delta}{(1+\gamma^t(1)/\sqrt{\Delta})^2}\right)
      \bigg]^{\Delta}\\
      &\le \kappa^2\, \left[ 1 + \frac{1}{\Delta}\E[(\gamma^t(0))^2] %
    +\frac{\kappa}{\Delta} \E\left( \frac{2\gamma^t(1) + 4\gamma^t(1)^2/\sqrt\Delta +%
    2 \gamma^t(1)^3/\Delta^{3/2}}{(1 + \gamma^t(1)/\sqrt\Delta)^2} \right) \right] ^\Delta \\
    &\le \kappa^2\,\left(1+\frac{1}{\Delta}\,\E[\gamma^t(0)^2]+
    \frac{2\kappa}{\Delta}\E[\gamma^t(1)]\right)^{\Delta}\\
    &\le\kappa^2\,\exp\Big[\E(\gamma^t(0)^2)+ 2\kappa\,\E(\gamma^t(1)) \Big]\\
    &\le\kappa^2\,\exp\Big[3\kappa\,\E(\gamma^t(1)) \Big]\, .
  \end{aligned}
\end{align*}
Consider, now, the third moment of $\gamma^{t+1}(1)$. Proceeding
in the same fashion as above we obtain that:
\begin{align*}
  \begin{aligned}
  \E\left[ \gamma^{t+1}(1)^3 \right] &= \kappa^3 \,\E\prod_{\ell=1}^\Delta \left( 1 + %
  \frac{\tA_\ell \gamma^t(x_\ell)/\sqrt\Delta}{1+\gamma^t(x_\ell)/\sqrt\Delta} \right)^3\\
  &\le \kappa^3 \Bigg( 1 + \frac{3\tkappa}{\Delta}\E\left[ \gamma^t(1) \right]
  + \frac{3}{\Delta} \E\left[ (\gamma^t(0))^2 \right] \\
  &\qquad+ \frac{\tkappa}{\sqrt\Delta}\E\left( \frac{3\gamma^t(1)^2/\Delta+ 4\gamma^t(1)^3/\Delta^{3/2}}{(1 + \gamma^t(1)/\sqrt\Delta)^3} \right)\Bigg)^\Delta. \\
\end{aligned}
\end{align*}
Since $(3z^2 + 4z^3)/(1+z)^3 \le 4z$ when $z\ge 0$, and that $\kappa>\tkappa$, 
we can bound the last term above to get:
\begin{align*}
  \E\left[ \gamma^{t+1}(1)^3 \right]  &\le \kappa^3\, %
  \left(  1 + \frac{3}{\Delta} \E\left[ \gamma^t(0)^2 \right] + 7\frac{\kappa}{\Delta} \E\left[ \gamma^t(1) \right]\right)^\Delta\\
  &\le \kappa^3\, \exp\left( 3\E[\gamma^t(0)^2]  + 7 \kappa \, \E[\gamma^t(1)]\right) \\
  &\le \kappa^3\, \exp\left(10\kappa\,\E[\gamma^t(1)]\right).
\end{align*}
The bound for the third moment of $\tgamma^{t+1}(1)$ follows from the 
same argument with the inclusion of the $(\Delta+1)^{\rm th}$ factor.
\end{proof}

\begin{lemma}
  \label{lem:logRawMoments}
  Consider $\gamma^t(0), \gamma^t(1), \tgamma^t(0), \tgamma^t(1)$
  satisfying the recursions
  Eqs.~\eqref{eq:FirstTreeRecursion}, \eqref{eq:SecondTreeRecursion},
  \eqref{eq:FirstVertexRecursion} and \eqref{eq:SecondVertexRecursion}.
  Also let $x$ be Bernoulli with parameter $\tkappa/\sqrt\Delta$ and 
  $A$ and $\tA$ be defined like $A^t_\ell$ and $\tA^t_\ell$ as in
  Eqs.~\eqref{eq:FirstTreeRecursion}, \eqref{eq:SecondTreeRecursion}. 
  Then we have that, for each $t \ge 0$
  and $m \in \{1, 2, 3\}$:
  \begin{align*}
    \E\Bigg[ \left\lvert\log\left(1 + \frac{A\gamma^t(x)/\sqrt\Delta}{1 + \gamma^t(x)\sqrt\Delta}\right) \right\rvert^m \Bigg]%
    &\le \frac{2\,\E[\gamma^t(x)^m]}{\Delta^{m/2}}\, ,  \\
    \E\Bigg[ \left\lvert\log\left(1 + \frac{\tA\gamma^t(x)/\sqrt\Delta}{1 + \gamma^t(x)\sqrt\Delta}\right) \right\rvert^m \Bigg]%
    &\le \frac{2\, \E[\gamma^t(x)^m]}{\Delta^{m/2}} \, . \\
  \end{align*}
\end{lemma}
\begin{proof}
  Consider the first claim. We have:
  \begin{align*}
    &\begin{aligned}
    \E\Bigg[ \left\lvert\log\left(1 + \frac{A\gamma^t(x)/\sqrt\Delta}{1 + \gamma^t(x)\sqrt\Delta}\right) \right\rvert^3 \Bigg]%
     &= \E\Bigg[\frac{1}{2}\left\lvert\log\left( 1 - \frac{\gamma^t(x)/\sqrt\Delta}{1 + \gamma^t(x)/\sqrt\Delta}\right)\right\rvert^m%
    \\&\qquad+ \frac{1}{2}\left\lvert\log\left( 1 + \frac{\gamma^t(x)/\sqrt\Delta}{1+\gamma^t(x)/\sqrt\Delta} \right)\right\rvert^m\Bigg] .
    &\end{aligned}
  \end{align*}
  Bounding the first term we get:
  \begin{align*}
    \E\Bigg[ \left\lvert \log\left( 1 - \frac{\gamma^t(x)/\sqrt\Delta}{1+ \gamma^t(x)\Delta} \right)\right\rvert^m \Bigg] &=%
    \int_0^\infty x\,\P\left( \log\left( 1 - \frac{\gamma^t(x)\sqrt\Delta}{1 + \gamma^t(x)/\sqrt\Delta} \right) \le - x^{1/m} \right)\md x\\
      &\le \int_0^\infty x\,\P\left(\gamma^t(x)/\sqrt\Delta \ge e^{x^{1/m}} - 1\right) \md x\\
      &\le \frac{\E\left[ \gamma^t(x)^m \right]}{\Delta^{m/2}}\int_0^\infty \frac{x}{(e^{x^{1/m}} - 1)^m} \md x \\
      &\le \frac{3\,\E\left[ \gamma^t(x)^m \right]}{\Delta^{m/2}}.
    \end{align*}
  For the second term, using $\log(1+z) \le z$ for $z\ge 0$ and the
  positivity of $\gamma^t(x)$:
  \begin{align*}
    \E\Bigg[\left\lvert\log\left( 1 + \frac{\gamma^t(x)/\sqrt\Delta}{1+\gamma^t(x)/\sqrt\Delta} \right)\right\rvert^m\Bigg] %
    &\le \frac{\E\left[ \gamma^t(x)^m \right]}{\Delta^{m/2}}. \\
  \end{align*}
  The combination of these yields the first claim.
  For the second claim, we write:
  \begin{align*}
    &\begin{aligned}
    \E\Bigg[ \left\lvert\log\left(1 + \frac{\tA\gamma^t(x)/\sqrt\Delta}{1 + \gamma^t(x)\sqrt\Delta}\right) \right\rvert^m \Bigg]%
    &\le \left( 1- \frac{\tkappa}{\sqrt\Delta} \right) %
      \E\Bigg[ \left\lvert\log\left(1 + \frac{A\gamma^t(0)/\sqrt\Delta}{1 + \gamma^t(0)\sqrt\Delta}\right) \right\rvert^m \Bigg]%
      \\&\qquad+ \frac{\kappa}{\sqrt\Delta}%
      \E\Bigg[ \left\lvert\log\left(1 + \frac{\gamma^t(1)/\sqrt\Delta}{1 + \gamma^t(1)\sqrt\Delta}\right) \right\rvert^m \Bigg]\\
      &\le \left( 1- \frac{\tkappa}{\sqrt\Delta} \right)\frac{2\, \E[\gamma^t(0)^m]}{\Delta^{m/2}} %
      + \frac{\tkappa}{\sqrt\Delta}\frac{\E[\gamma^t(1)^m]}{\Delta^{m/2}}\\
      &\le \frac{2\, \E[\gamma^t(x)^m]}{\Delta^{m/2}}.
    \end{aligned}
  \end{align*}
  Here the penultimate inequality follows in the same fashion as 
  for the first claim. 
\end{proof}
%
%
\subsection{Proof of Lemma  \ref{lemma:TreeSmallKappa}}

For $\kappa<1/\sqrt{e}$, the recursive bounds in
Eq.~(\ref{lem:gammaRawMoments})
yield bounds on the first three moments of $\gamma^t(1)$ that are
uniform in $t$. Precisely, we have the following:
\begin{lemma}\label{lem:MomentMessages}
  For $\kappa< 1/\sqrt{e}$, let $\gamma_*=\gamma_{\star}(\kappa)$ be the smallest positive solution of the equation
\begin{align*}
\gamma &= \kappa\, e^{\kappa\gamma}\, .
\end{align*}
Then we have, for all $t\ge 0$:
\begin{align}
  \E\gamma^t(1) &\le \gamma_* \, ,\label{eq:FirstMomentMessSmall1}\\
  \E(\gamma^t(1)^2) &\le \frac{\gamma_*^3}{\kappa}\, , \label{eq:SecondMomentMessSmall1} \\
  \E(\gamma^t(1)^3) &\le \frac{\gamma_*^{10}}{\kappa^7} \,\label{eq:ThirdMomentMessSmall1}.
\end{align}
Moreover, we have for all $t\ge 0$:
\begin{align}
  \E\tgamma^t(1) &\le \gamma_*\left(1 + \frac{\kappa\gamma_*}{\Delta} \right) \,,\label{eq:FirstMomentVertMess1}\\
  \E(\tgamma^t(1)^2) &\le \frac{\gamma_*^3}{\kappa}\left( 1 +
    \frac{3\kappa\gamma_*}{\Delta} \right) \, ,\label{eq:SecondMomentVertMess1} \\
  \E(\tgamma^t(1)^3) &\le \frac{\gamma_*^{10}}{\kappa^7}\left(1 +
    \frac{10\kappa\gamma_*}{\Delta}  \right)\, . \label{eq:ThirdMomentVertMess1}.
\end{align}
\end{lemma}
\begin{proof}
  We need only prove \myeqref{eq:FirstMomentMessSmall1} since
  the rest follow trivially from it and Lemma \ref{lem:gammaRawMoments}.
  The claim of \myeqref{eq:FirstMomentMessSmall1} follows from induction 
  and Lemma \ref{lem:gammaRawMoments} since
  $E[\gamma^{0}(1)]=\kappa<\tkappa\le \gamma_*$, and noting that, for
  $\gamma<\gamma_*$, $\tkappa\exp(\tkappa\,\gamma)<\gamma_*$. 
\end{proof}
The following is a simple consequence of the central limit theorem.
\begin{lemma}\label{lem:CLT}
For any $a<b\in\reals$, $\sigma^2>0$, $\rho<\infty$, there exists
$n_0=n_0(a,b,\sigma^2,\rho)$
such that the following holds for all $n\ge n_0$.
Let $\{W_i\}_{1\le i\le n}$ be i.i.d. random variables, with
$\E\{W_1\}\ge a/n$, $\Var(W_1)\ge \sigma^2/n$ and $\E\{|W_1|^3\}\le \rho/n^{3/2}$. Then
\begin{align*}
\prob\Big\{\sum_{i=1}^nW_i\ge b\Big\}\ge
\frac{1}{2}\, \Phi\Big(-\frac{b-a}{\sigma}\Big)\, .
\end{align*}
\end{lemma}
\begin{proof}
Let $a_0 = n\E\{W_2\}$ and $\sigma_0^2 \equiv n\Var(W_1)$.
By the Berry-Esseen central limit theorem, we have
\begin{align*}
\prob\Big\{\sum_{i=1}^nW_i\ge b\Big\}&\ge
\Phi\Big(-\frac{b-a_0}{\sigma_0}\Big)-\frac{\rho}{\sigma_0^3\sqrt{n}}\\
&\ge \Phi\Big(-\frac{b-a}{\sigma}\Big)-\frac{\rho}{\sigma^3\sqrt{n}}
\end{align*}
The claim follows by taking $n_0\ge \rho^2/[\sigma^3\Phi(-u)]^2$ with
$u\equiv (b-a)/\sigma$.
\end{proof}

We finally prove a statement that is stronger than Lemma
\ref{lemma:TreeSmallKappa}, since it also controls  the distribution
of $\tgamma^t(0)$.
\begin{proposition}
Assume $\kappa<1/\sqrt{e}$ and let $\gamma_*$ be defined as per Lemma
\ref{lem:MomentMessages}. Then there exists $\Delta_* =
\Delta_*(\kappa)$, $\delta_*= \delta_*(\kappa)>0$  such that, for all $\Delta>\Delta_*(\kappa)$ and all
$t\ge 0$, we have:
\begin{align*}
 \P(\tgamma^t(0)\ge 5\gamma_*) &\ge \delta_* ,\\
 \P(\tgamma^t(1)\le 5\gamma_*)&\ge \frac{3}{4}\, .
\end{align*}
where $\delta_*$ is given by:
\begin{align*}
  \delta_* &= \frac{1}{2}\Phi\left(8 \frac{2\kappa -
      \log(5\gamma_*/\kappa)} {(\kappa^3/\gamma_*)^{1/2}}  \right)\,
  .
\end{align*}
where $\Phi(\,\cdot\,)$ is the cumulative distribution function of the standard
normal.
\end{proposition}
\begin{proof}
The second bound follows from Markov inequality and Lemma
\ref{lem:MomentMessages} for large enough $\Delta$. 
As for the first one, we have using $\Gamma^t(0) = \log\tgamma^t(0)$:
\begin{align*}
  \Gamma^{t+1}(0) &= \log\kappa+\sum_{\ell=1}^{\Delta+1}\log\left(1+ \frac{A_\ell\gamma^t_\ell(x_\ell)/\sqrt\Delta}{1+\gamma^t(x_\ell)/\sqrt\Delta}\right)\, ,
\end{align*}
  
  It follows from Lemma \ref{lem:logRawMoments} and Lemma \ref{lem:MomentMessages}
  that:
  \begin{align*}
    \E\left[\log\left(1+ \frac{A_\ell\gamma^t_\ell(x_\ell)/\sqrt\Delta}{1+\gamma^t(x_\ell)/\sqrt\Delta}\right)\right]%
    &\ge - \frac{2\kappa}{\sqrt\Delta}.
  \end{align*}

  We now lower bound the variance of each summand using the conditional variance given 
  $\gamma^t(x)$. Since $A \in \{\pm1\}$ are independent of $\gamma^t(x)$ and uniform, we
  have:
  \begin{align*}
  \Var\left[\log\left(1+ \frac{A_\ell\gamma^t_\ell(x_\ell)/\sqrt\Delta}{1+\gamma^t(x_\ell)/\sqrt\Delta}\right)\right] %
  &\ge \E\left[ \Var\left( \log\left(1+ \frac{A_\ell\gamma^t_\ell(x_\ell)/\sqrt\Delta}{1+\gamma^t(x_\ell)/\sqrt\Delta}\right) \Bigg\lvert \gamma^t(x) \right) \right] \\
    &= \frac{1}{4}\E\Bigg[\left( \log\left(1 + \frac{2\gamma^t(x)}{\sqrt\Delta}\right) \right)^2\Bigg]\\
    &\ge \frac{1}{4}\left(\log\left( 1+\frac{\kappa}{\sqrt\Delta} \right)\right)^2 \P(\gamma^t(x) \ge \kappa/2).
  \end{align*}
  We have, using Lemma \ref{lem:MomentMessages} that:
  \begin{align*}
    \E[\gamma^t(x)] &\ge \kappa\\
    \E[\gamma^t(x)^2] &\le 2\gamma_*\kappa.
  \end{align*}
  Using the above and the Paley-Zygmund inequality, we get:
  \begin{align*}
    \Var\left[\log\left(1+ \frac{A_\ell\gamma^t_\ell(x_\ell)/\sqrt\Delta}{1+\gamma^t(x_\ell)/\sqrt\Delta}\right)\right] %
    &\ge \frac{1}{32\gamma_*}\left[\log\left( 1+\frac{\kappa}{\sqrt\Delta} \right)\right]^2 \\
    &\ge \frac{\kappa^3}{64\gamma_*\Delta},
  \end{align*}
  for $\Delta$ large enough. 
Now, employing Lemma \ref{lem:CLT} we get:
\begin{align*}
  \P(\tgamma^t(0)\ge 5\gamma_*) \ge \frac{1}{2}\Phi\left( -8 \frac{\log(5\gamma_*/\kappa) - 2\kappa }{(\kappa^3/\gamma_*)^{1/2}} \right).
\end{align*}
\end{proof}
%
%
\def\sF{{\sf F}}
\def\Le{{\sf Le}}
\def\omu{\overline{\mu}}

\subsection{Proof of Lemma  \ref{lemma:TreeLargeKappa}}

As discussed in Section \ref{sec:Sparse}, BP minimizes the
misclassification error  at vertex $i$ among all $t$-local algorithms
provided $\Ball_{G_N}(i;t)$ is a tree. Equivalently it minimizes the
misclassification rate at the root of the regular tree $\tTree(t)$. 
We will prove Lemma  \ref{lemma:TreeLargeKappa} by proving that there
exists a local algorithm to estimate the root value on the tree
$\Tree(t)$, for a suitable choice of $t$ with error rate
$\exp(-\Theta(\sqrt{\Delta}))$. Note that since the labeled tree
$\Tree(t)$ is a subtree of $\tTree(t)$, the same algorithm is local on
$\Tree(t)$. Formally we have the following.
\begin{proposition}\label{propo:LargeKappaAll}
Assume $\kappa>1/\sqrt{e}$. Then there exists  $c_*=c_*(\kappa)>0$
 $\Delta_*=\Delta_*(\kappa)<\infty$, $t_*=t_* (\kappa,\Delta)<\infty$ such that,
for all $\Delta>\Delta_*(\kappa)$ we can construct a 
$t_*$-local decision rule $\sF:W_{\Tree(t_*)}\to \{0,1\}$
satisfying
\begin{align*}
\prob(\sF(W_{\Tree(t_*)})\neq x_\circ|X_{\circ} = x_{\circ})\le
e^{-c_*\sqrt{\Delta}}\, ,
\end{align*}
for $x_{\circ}\in\{0,1\}$.
\end{proposition}
The rest of this section is devoted to the proof of this proposition.

The decision rule $\sF$ is constructed as follows. We write
$t_*=t_{*,1}+t_{*,2}$ with $t_{*,1}$ and $t_{*,2}$ to be chosen below,
and decompose the tree $\Tree(t_*)$ into its first $t_{*,1}$
generations (that is a copy of $\Tree(t_{*,1})$) and its last
$t_{*,2}$ generations
(which consist of $\Delta^{t_{*,1}}$ independent copies of
$\Tree(t_{*,2})$). 
We then run a first decision rule based on BP for the copies of $\Tree(t_{*,2})$
that correspond to the last $t_{*,2}$ generations. This yields
decisions that have a small, but independent of $\Delta$, error
probability on the nodes at generation $t_{*,1}$.
We then refine these decisions by running a different algorithm on the
first $t_{*,1}$ generations. 

Formally, Proposition  \ref{propo:LargeKappaAll} follows from the
following two lemmas, that are proved next.
The first lemma provides the decision rule for the nodes at generation
$t_{*,1}$, based on the last $t_{*,2}$ generations.
\begin{lemma}\label{lemma:LargeKappaAll_FIRST}
Assume $\kappa>1/\sqrt{e}$ and let $\eps>0$ be arbitrary. Then there exists 
 $\Delta_*=\Delta_*(\kappa,\eps)<\infty$, $t_{*,2}=t_{*,2} (\kappa,\Delta)<\infty$ such that,
for all $\Delta>\Delta_*(\kappa,\eps)$ there exists a $t_{*,2}$-local decision rule 
$\sF_2:W_{\Tree(t_{*,2})}\mapsto \sF_2(W_{\Tree(t_{*,2})})\in\{0,1\}$
such that
\begin{align*}
\prob(\sF_2(W_{\Tree(t_{*,2} )})\neq x_\circ|X_{\circ} = x_{\circ})\le
\eps\, ,
\end{align*}
for $x_{\circ}\in \{0,1\}$.
\end{lemma}
The second lemma yields a decision rule for the root,
given information on the first $t_{*,1}$ generations, as well as
decisions on the nodes at generation $t_{*,1}$. In order to 
state the theorem, we denote by $\Le(t)$ the set of nodes 
at generation $t$. For $\eps=(\eps(0),\eps(1))$, we also let $Y_{\Le(t_{*,1})} (\eps) = (Y_i(\eps))_{i\in
  \Le(t_{*,1})}$ denote a collection of
random variables with values in $\{0,1\}$ that are independent given
the node labels  $X_{\Le(t_{*,1})} = (X_i)_{i\in
  \Le(t_{*,1})}$ and such that, for all $i$,
\begin{align*}
\prob\{Y_i(\eps) \neq X_i|X_i = x\} \le \eps(x)\, .
\end{align*}
\begin{lemma}\label{lemma:LargeKappaAll_SECOND}
Fix $\kappa\in (0,\infty)$.
There exists $\Delta_*(\kappa)=\Delta_*(\kappa)<\infty$,
$t_{*,1}<\infty$, $c_*(\kappa)>0$ and
$\eps_*(\kappa)>0$ such that. 
for all $\Delta>\Delta_*$ there exists a  
$t_{*,1}$-local decision rule 
\begin{align*}
\sF_1:(W_{\Tree(t_{*,1})},Y_{\Le(t_{*,1})})\mapsto
\sF_1(W_{\Tree(t_{*,1})} Y_{\Le(t_{*,1})})\in\{0,1\}\,
\end{align*}
satisfying, for any $\eps\le \eps_*$: 
\begin{align*}
\prob(\sF_1(W_{\Tree(t_{*,1})} Y_{\Le(t_{*,1})})\neq x_\circ|X_{\circ} = x_{\circ})\le
e^{-c_*\sqrt{\Delta}}\, ,
\end{align*}
\end{lemma}
%
%
\subsubsection{Proof of Lemma \ref{lemma:LargeKappaAll_FIRST}}

The decision rule is constructed as follows. 
We run BP on the tree $\Tree(t_{*,2})$ under consideration, thus
computing the likelihood ratio $\gamma^{t_{*,2}}$ at the
root. We then set $\sF_2(W_{\Tree(t_{*,2})}) = \ind(\gamma^{t_{*,2}}\ge e^{\omu})$. 
Here $\omu$ is a threshold to be chosen below.

In order to analyze this rule, recall that $\gamma^t(x_{\circ})$
denotes a random variable whose distribution is the same as the
conditional distribution of $\gamma^t$, given the true value of the
root $X_{\circ} = x_{\circ}$. It is convenient to define
\begin{align*}
\Gamma^t(x) &\equiv \log \gamma^t(x)\\
\hkappa &\equiv \log \kappa.
\end{align*}
As stated formally below, in the limit of large $\Delta$, 
$\Gamma^t(0)$ ($\Gamma^t(1)$) is asymptotically
Gaussian with mean $\mu_t(0)$ (resp. 
$\mu_t(1)$) and variance $\sigma_t^2$. The 
parameters $\mu_t(0), \mu_t(1), \sigma_t$ 
are defined by the recursion:
\begin{align}
  \mu_{t+1}(0) &= \hkappa - \frac{1}{2}e^{2\mu_t(0) + 2\sigma_t^2}\, ,\label{eq:stateevold0}\\
  \mu_{t+1}(1) &= \hkappa +  \kappa\,e^{\mu_t(1) + \sigma_t^2/2}%
   - \frac{1}{2}e^{2\mu_t(0) + 2\sigma_t^2} \, ,\label{eq:stateevold1}\\
  \sigma_{t+1}^2 &= e^{2\mu_t(0) + 2\sigma_t^2}\, , \label{eq:stateevold2}. 
\end{align}
with initial conditions $\mu_0(0) = \mu_0(1) = \hkappa$ and $\sigma_0^2=0$. 
Formally, we have the following:
\begin{proposition}\label{propo:GaussianConvergence}
  Fix a time $t > 0$. Then
  the following limits hold as $\Delta\to\infty$:
  \begin{align*}
    \Gamma^t(0) &\convD \mu_t(0) + \sigma_t Z  \\
    \Gamma^t(1) &\convD \mu_t(1) + \sigma_t Z
  \end{align*}
  where $Z\sim \normal(0, 1)$ and $\mu_t(0), \mu_t(1), \sigma_t$ are defined
  by the state evolution recursions (\ref{eq:stateevold0}) to (\ref{eq:stateevold2}). 
\end{proposition}
\begin{proof}
  We prove the claims by induction.  
  We have for $1\le i\le t-1$:
  \begin{align*}
    \Gamma^{i+1}(0) &= \hkappa + \sum_{\ell=1}^\Delta\log\left( 1+%
    \frac{A^i_\ell\gamma^i_\ell(x_\ell)/\sqrt{\Delta}}{1+\gamma^i_\ell(x_\ell)/\sqrt{\Delta}} \right).
  \end{align*}
  Considering the first moment, we have:
  \begin{align*}
    \E\left[ \Gamma^{i+1}(0) \right] &= \hkappa + \E\left[\Delta \log\left( 1+%
    \frac{A^i_\ell\gamma^i(x)/\sqrt{\Delta}}{1+\gamma^i(x)/\sqrt{\Delta}} \right) \right]\\
    &= \hkappa + \E\left[ \frac{\Delta}{2}\log\left( 1 + \frac{\gamma^i(x)/\sqrt{\Delta}}{1+\gamma^i(x)/\sqrt{\Delta}} \right) + 
    \frac{\Delta}{2}\log \left( 1 - \frac{\gamma^i(x)/\sqrt{\Delta}}{1+\gamma^i(x)/\sqrt{\Delta}} \right)\right].
  \end{align*}
  where we drop the subscript $\ell$. Expanding similarly using the distribution of $x$, we find that
  the quantity inside the expectation converges
  pointwise as $\Delta\to\infty$ to:
  \begin{align*}
    -\gamma^i(0)^2 &= -e^{2\Gamma^i(0)}.
  \end{align*}
  Since $\E[\gamma^i(0)^2]$ is bounded by Lemma \ref{lem:gammaRawMoments} for fixed $t$,
  we have by dominated convergence and the induction hypothesis that:
  \begin{align*}
    \lim_{\Delta\to\infty} \E\left[ \Gamma^{i+1}(0) \right] &= \hkappa - \frac{1}{2}e^{2\mu_i(0) + 2\sigma_i^2}\\
    &= \mu_{i+1}(0).
  \end{align*}
  Similarly, for the variance we have:
  \begin{align*}
    \Var(\Gamma^{i+1}(0)) &= \Var\left( \sqrt{\Delta}\log\left( 1+ \frac{A^i\gamma^i(x)/\sqrt\Delta}{1 + \gamma^i(x)/\sqrt{\Delta}} \right) \right).
  \end{align*}
  Using $\Var(X) = \E(X^2) - (\E X)^2$, we find that the right hand side 
  converges pointwise to $\gamma^i(0)^2 = e^{2\Gamma^i(0)}$, yielding as 
  before:
  \begin{align*}
    \lim_{\Delta\to\infty}\Var(\Gamma^{i+1}(0)) &= e^{2\mu_i(0) + 2\sigma_i^2} \\
    &= \sigma_{i+1}^2.
  \end{align*}
  It follows from the Lindeberg central limit theorem that $\Gamma^{i+1}$
  converges in distribution to $\mu_{i+1}(0) + \sigma_{i+1}Z$ where
  $Z\sim\normal(0, 1)$.
  For the base case, $\Gamma^0(0)$ is initialized with the value
  $\log \kappa = \hkappa$.
  This is trivially the (degenerate) Gaussian given by $\mu_0(0) +\sigma_0 Z$
  since $\mu_0(0)=\hkappa$ and $\sigma_0=0$.

  Now consider the case of $\Gamma^{i+1}(1)$. We have as before:
  \begin{align*}
    \Gamma^{i+1}(1) &= \hkappa + \sum_{\ell=1}^\Delta\log\left( 1+\frac{\tA^i_\ell\gamma^i_\ell(x_\ell)/\sqrt{\Delta}}{1 + \gamma^i_\ell(x_\ell)/\sqrt{\Delta}} \right).
  \end{align*}
  Computing the first moment:
  \begin{align*}
    \begin{aligned}
    \E[\Gamma^{i+1}(1)] &= \hkappa + \Delta\,\E \left[ \log\left( 1+\frac{\tA^i\gamma^i(x)/\sqrt{\Delta}}{1 + \gamma^i(x)/\sqrt{\Delta}} \right) \right] \\
    &= \hkappa + \Delta\, \E \Bigg[\frac{\tkappa}{\sqrt{\Delta}} \log\left( 1+\frac{\gamma^i(1)/\sqrt{\Delta}}{1 + \gamma^i(1)/\sqrt{\Delta}} \right)\\%
    &\qquad+ \left(1 - \frac{\tkappa}{\sqrt{\Delta}}\right) \log\left( 1+\frac{A^i\gamma^i(0)/\sqrt{\Delta}}{1 + \gamma^i(0)/\sqrt{\Delta}} \right)\Bigg]. 
  \end{aligned}
  \end{align*}
  The second term can be handled as before. The first term in the expectation
  converges pointwise to $\kappa\gamma^i(1) = \kappa e^{\Gamma^i(1)}$. Thus
  we get by dominated convergence as before that:
  \begin{align*}
    \lim_{\Delta\to\infty}\E[\Gamma^{i+1}(1)] &= \hkappa + \kappa\,e^{\mu_i(1) + \sigma_i^2/2} - \frac{1}{2}e^{2\mu_i(0) + 2\sigma_i^2}\\
    &= \mu_{i+1}(0) + \kappa\, e^{\mu_i(1) + \sigma_i^2/2}\\
    &= \mu_{i+1}(1).
  \end{align*}
  For the variance:
  \begin{align*}
    \Var\left( \Gamma^{i+1}(1) \right) &= \Delta \,\Var\left( \log\left( 1+\frac{\tA^i\gamma^i(x)/\sqrt{\Delta}}{1 + \gamma^i(x)/\sqrt{\Delta}} \right) \right).
  \end{align*}
  Dealing with each term of the variance computation separately we get:
  \begin{align*}
    &\begin{aligned}
    \Delta\, \E\left[\log^2\left( 1+\frac{\tA^i\gamma^i(x)/\sqrt{\Delta}}{1 + \gamma^i(x)/\sqrt{\Delta}} \right)\right] %
    &= \Delta\,\E\bigg[ \frac{\tkappa}{\sqrt{\Delta}} \log^2\left( 1+\frac{\gamma^i(1)/\sqrt{\Delta}}{1 + \gamma^i(1)/\sqrt{\Delta}} \right)\\%
     &\qquad + \left( 1-\frac{\tkappa}{\sqrt{\Delta}} \right) \log^2\left( 1+\frac{A^i\gamma^i(0)/\sqrt{\Delta}}{1 + \gamma^i(0)/\sqrt{\Delta}}\right) \bigg], \\
  \end{aligned}
  \end{align*}
  Asymptotically in $\Delta$, the contribution of first term vanishes, while that of the other
  can be computed, using dominated convergence as in the case of $\Gamma^{i+1}(0)$, as:
  \begin{align*}
    \lim_{\Delta\to\infty}\E\left[ e^{2\Gamma^i(0)} \right] &= e^{2\mu_i(0) + 2\sigma_i^2}\\
    &= \sigma_{i+1}^2
  \end{align*}
  where we use the induction hypothesis. Similarly expanding:
  \begin{align*}
    \begin{aligned}
    \E\left[\sqrt{\Delta}\log\left( 1+\frac{\tA^i\gamma^i(x)/\sqrt{\Delta}}{1 + \gamma^i(x)/\sqrt{\Delta}} \right)\right] %
    & = \sqrt{\Delta}\E\bigg[ \frac{\tkappa}{\sqrt{\Delta}}\log\left( 1+ \frac{\gamma^i(1)/\sqrt{\Delta}}{1 + \gamma^i(1)/\sqrt{\Delta}} \right)\\%
    &\qquad +  \left(1-\frac{\tkappa}{\sqrt{\Delta}}\right) \log \left( 1 + \frac{A^i \gamma^i(0)/\sqrt{\Delta}}{1+\gamma^i(0)/\sqrt{\Delta}} \right)\bigg].
    \end{aligned}
  \end{align*}
  In this case, the contribution of both terms goes to zero, hence 
  asymptotically in $\Delta$ the expectation above vanishes.
  It follows, using these computations and the Lindeberg central limit 
  theorem that $\Gamma^{i+1}(1)$ converges in distribution to the limit
  random variable $\mu_{i+1}(1) + \sigma_{i+1}Z$ where $Z\sim\normal(0, 1)$.
  The base case for $\Gamma^{i+1}(1)$ is the same as that for $\Gamma^{i+1}(0)$
  since they are initialized at the (common) value $\hkappa$.
\end{proof}

Using the last lemma we can estimate the probability of error 
that is achieved by thresholding the likelihood ratios $\gamma^t$.
\begin{corollary}\label{coro:F2t}
Let $\gamma^t$ be the likelihood ratio at the root of tree $\Tree(t)$.
Define $\omu_t \equiv(\mu_t(1)+\mu_t(0))/2$, and set
$\sF_{2,t}(W_{\Tree(t)})=1$ if $\gamma^t>e^{\omu_t}$ and $\sF_{2,t}(W_{\Tree(t)})=0$
otherwise. Then, there exists $\Delta_0(t)$ such that, for all $\Delta>\Delta_0(t)$,
\begin{align*}
\prob(\sF_{2,t}(W_{\Tree(t)})\neq x_{\circ}|X_{\circ} =
x_{\circ})\le 2\,\Phi\Big(-\frac{\mu_t(1)-\mu_t(0)}{2\sigma_t}\Big)\, ,
\end{align*}
\end{corollary}
\begin{proof}
It is easy to see, from Eqs.~(\ref{eq:stateevold0}) and
(\ref{eq:stateevold1}) that $\mu_t(1)\ge \mu_t(0)$. 
By the last lemma, 
\begin{align*}
\lim_{\Delta\to\infty}\prob(\sF_{2,t}(W_{\Tree(t)})\neq x_{\circ}|X_{\circ} =
x_{\circ})&= \,\Phi\Big(-\frac{\mu_t(1)-\mu_t(0)}{\sigma_t}\Big)\, ,
\end{align*}
and this in turn yields the claim.
\end{proof}

Finally, we have the following simple calculus lemma, whose proof we
omit.
\begin{lemma}
Let $\mu_t(0)$, $\mu_t(1)$, $\sigma_t$ be defined as per Eqs.~(\ref{eq:stateevold0}),
(\ref{eq:stateevold1}), (\ref{eq:stateevold2}). If $\kappa
>1/\sqrt{e}$, then 
\begin{align*}
\lim_{t\to\infty} \, \frac{\mu_t(1)-\mu_t(0)}{\sigma_t} = \infty\, .
\end{align*}
\end{lemma}

Lemma \ref{lemma:LargeKappaAll_FIRST} follows from combining this
result with Corollary \ref{coro:F2t} and selecting $t_{*,2}$ so that
$\Phi\big(-(\mu_t(1)-\mu_t(0))/\sigma_t\big)\le \eps/4$ for $t=
t_{*,2}$. Finally we let $\Delta_* = \Delta_0(t_{*,2})$ as per
Corollary \ref{coro:F2t}.

\subsubsection{Proof of Lemma \ref{lemma:LargeKappaAll_SECOND}}

\def\tm{\widetilde{m}}
\def\Bin{{\sf Bin}}

We construct the decision rule $\sF_1:(W_{\Tree(t_{*,1})},Y_{\Le(t_{*,1})})\mapsto
\sF_1(W_{\Tree(t_{*,1})} Y_{\Le(t_{*,1})})\in\{0,1\}$ recursively as
follows. For each vertex $i\in \Tree(t)$ we compute a decision
$m_i\in\{0,1\}$ based on the set of descendants of $i$, do be denoted
as $D(i)$, as follows 
\begin{align}
  m^{t+1}_{i} &= \begin{cases}
    1 \text{ if } \sum_{\ell\in D(i)}W_{i\ell}\, m^t_{\ell} \ge \frac{\kappa}{2}\sqrt{\Delta} \\
    0 \text{ otherwise.}
  \end{cases}\label{eq:cleanMP1} 
\end{align}
If $i$ is a leaf, we let $m_i = Y_i$. Finally we take a decision 
on the basis of the value at the root:
\begin{align*}
  F_1(W_{\tTree(t)},Y_{\Le(t)})= m_{\root}\, .
\end{align*}

We  recall that the  $Y_i$'s are conditionally independent given
 $X_{\Le(t)}$. We assume  that $\prob(Y_i = 1|X_i = 0) =\prob(Y_i =
 0|X_i = 1)=\eps$, which does not entail any loss of generality
 because it can always be achieved by degrading the decision rule $\sF_2$.
We define the following quantities:
\begin{align}
  p_t &= \P(F_1(W_{\tTree(t)},Y_{\Le(t)})=1\rvert X_\root = 0) \, ,\label{eq:ptdef}\\
  q_t &= \P(F_1(W_{\tTree(t)},Y_{\Le(t)})=1\rvert X_i = 1). 
  \label{eq:qtdef}
\end{align}
Note in particular that $p_0=1-q_0=\eps$.

Lemma  \ref{lemma:LargeKappaAll_SECOND} follows immediately from the following.
\begin{lemma}
  Let $p_t$ and $q_t$ be defined as in Eqs.~\eqref{eq:ptdef}, \eqref{eq:qtdef}.
  Then there exists $\eps_* = \eps_*(\kappa)$ and  $t_* =
  t_*(\Delta,\kappa)<\infty$, 
 $\Delta_*=\Delta_*(\kappa)<\infty$, $c_*=c_*(\kappa)>0$ such that for $\Delta\ge \Delta_*$:
  \label{lem:ptbehavior}
\begin{align}
  p_{t_*} &\le 4\,e^{-c_*\sqrt{\Delta}} \\
  1-q_{t_*} &\le 4\,e^{ -c_*\sqrt\Delta}.
\end{align}
\end{lemma}
\begin{proof}
  Throughout the proof, we use $c_1,c_2,\dots$ to denote constants
  that can depend on $\kappa$ but not on $\Delta$ or $t$.
  We first prove the following by induction:
  \begin{align*}
    p_{t+1} &\le e^{- c_1\Delta p_t} + e^{- c_2\Delta^{3/2}} + e^{-c_3/p_t}\\
    1 - q_{t+1} &\le e^{- c_1\Delta p_t} + e^{- c_2\Delta^{3/2}} + e^{-c_3/4p_t}.
  \end{align*}
  
  We let $\Bin(n, p)$ denote the binomial distribution with
  parameters $n, p$. First, let
  $D \sim \Bin(\Delta, \tkappa/\sqrt{\Delta})$ and, conditional on $D$:
  \begin{align*}
    N_t &\sim \Bin(\Delta - D, p_t) \\
    M_t &\sim \Bin(D, q_t). 
  \end{align*}
  From the definition of $p_t$, $q_t$ and Eq.~(\ref{eq:cleanMP1}) we observe that:
  \begin{align*}
    p_{t+1} &= \P\left(\sum_{i = 1}^{M_t + N_t} W_i \ge \kappa\sqrt{\Delta}/2\right)\\
    q_{t+1} &= \P\left(\sum_{i = 1}^{N_t} W_i + M_t \ge \kappa\sqrt{\Delta}/2\right),
  \end{align*}
  where we let $W_i\in\{\pm 1\}$ are i.i.d and uniformly
  distributed.
  
  Considering the $p_t$ recursion:
  \begin{align*}
    p_{t+1} &\le \P\left( M_t + N_t \le \frac{p_t\Delta}{2} \right) + %
    \P\left( \sum_{i=1}^{p_t\Delta/2}W_i \ge \frac{\kappa\sqrt{\Delta}}{2}  \right)  \\
    &\le \P\left( N_t \le \frac{p_t\Delta}{2}\right) %
    + e^{-\kappa^2/4p_t} \\
    &\le \P\left( \tilde N \le \frac{p_t\Delta}{2} \right) + \P\left( D\ge \frac{\Delta}{4} \right) + e^{-\kappa^2/4p_t},
  \end{align*}
  where $\tilde N \sim \Bin(3\Delta/4, p_t)$ and the penultimate inequality follows
  from the Chernoff bound and positivity of $M_t$. Using standard Chernoff bounds, 
  for $\Delta \ge 5$ we get:
  \begin{align*}
    p_{t+1} &\le e^{- 3\Delta p_t/128} + e^{- c_1\Delta^{3/2}} + e^{-\kappa^2/4p_t},
  \end{align*}
  where $c_1$ is a constant dependent only on $\kappa$. 
  Bounding $1-q_{t+1}$  in a similar fashion, we obtain:
  \begin{align*}
    1 - q_{t+1} &= \P\left( \sum_{i = 1}^{N_t}W_i + M_t \le \kappa\sqrt{\Delta}/2 \right) \\
    &\le \P\left( \sum_{i=1}^{N_t} W_i \le -\kappa \sqrt{\Delta}/4 \right) +  \P\left( M_t \le 3\kappa\sqrt{\Delta}/4  \right).
  \end{align*}
  We first consider the $M_t$ term:
  \begin{align*}
    \P\left( M_t \le 3\kappa\sqrt{\Delta}/4 \right) &\le \P\left( D\le 7\kappa\sqrt{\Delta}/8 \right) %
    + \P\left( \tilde M \ge \kappa\sqrt{\Delta}/8 \right),
  \end{align*}
  where $\tilde M \sim \Bin(0, q_t)$. Further, using Chernoff bounds we
  obtain that this term is less than $2e^{-\kappa\sqrt{\Delta}/128}$. 
  The other term is handled similar to the $p_t$ recursion as:
  \begin{align*}
    \P\left( \sum_{i=1}^{N_t} W_i \le -\kappa\sqrt{\Delta}/4 \right) &\le %
    \P\left( N_t \le \frac{p_t\Delta}{2} \right) + \P\left( \sum_{i=1}^{p_t\Delta/2} W_i \ge \kappa\sqrt{\Delta}/4 \right) \\
    &\le \P\left( \tilde N \le \frac{p_t\Delta}{2} \right) + \P\left( D\ge \frac{\Delta}{4} \right) + e^{-\kappa^2/16p_t} \\
    &\le e^{-3\Delta p_t/128} + e^{-c_1\Delta^{3/2}} + e^{-\kappa^2/16p_t}.
  \end{align*}
  Consequently we obtain:
  \begin{align*}
    1-q_t &\le e^{-3\Delta p_t/128} + e^{-c_1\Delta^{3/2}} + e^{-\kappa^2/16p_t} + 2e^{-\kappa\sqrt{\Delta}/128}.
  \end{align*}

Simple calculus shows that, for all $\Delta>\Delta_*(\kappa)$ large
enough, there exists $\eps_*$, $c_0$ dependent on
$\kappa$ but independent of
$\Delta$ such that 
\begin{align*}
e^{- 3\Delta p/128} + e^{-c_1 \Delta^{3/2}} + e^{-\kappa^2/16p}< p
\end{align*}
for all $p\in [c_0/\sqrt{\Delta},\eps_*]$. Since $p_0=\eps\le \eps_*$,
this implies that there exists $t_0$ such that $p_{t_0}\le
c_0/\sqrt{\Delta}$.
The claim follows by taking $t_* = t_0+1$ which yields that
$p_{t_*} = O(e^{-\Theta(\sqrt{\Delta})})$. Further, observing that
the error $1-q_{t_*}$ has only an additional $2e^{-c_3\sqrt{\Delta}}$
component, we obtain a similar claim for $1-q_{t_*}$.
\end{proof}

\section*{Acknowledgments}

This work was partially supported by the NSF CAREER award CCF-0743978, the NSF grant DMS-0806211,
and the grants  AFOSR/DARPA FA9550-12-1-0411 and FA9550-13-1-0036.

\appendix

\section{Some tools in probability theory}
\label{app:Tools}

This appendix contains some useful facts in probability theory.

\begin{lemma}
	Let $h:\reals\to\reals$ be a bounded function with first three derivatives uniformly bounded.
	Let $X_{n, k}$ be mutually independent random variables for $1\le k\le n$ with zero mean and 
	variance $v_{n, k}$.
	Define:
	\begin{align*}
		v_n &\equiv \sum_{k=1}^n v_{n, k} \\
		\delta_n(\eps) &\equiv \sum_{k=1}^n \E[X^2_{n, k}\ind_{\lvert X_{n, k}\rvert \ge \eps}]\\
		S_n &\equiv \sum_{k=1}^n X_{n, k}.
	\end{align*}
	Also let $\G_n = \normal(0, v_n)$. Then, for every $n$ and $\eps>0$:
	\begin{align*}
		\lvert\E h(S_n) - \E h(G_n)\rvert \le \left( \frac{\eps}{6} + \frac{\sqrt{\eps^2 + \delta_n}}{2}\right)%
		v_n\lVert h'''\rVert_\infty + \delta_n\lVert h'' \rVert_\infty 
	\end{align*}
	\label{lem:lindeberg}
\end{lemma}
\begin{proof}
  The lemma is proved using a standard swapping trick. The
   proof can be found in Amir Dembo's lecture notes \cite{DemboNotes}.
\end{proof}

\begin{lemma}\label{lem:subgauss}
	Given a random variable $X$ such that $\E(X) = \mu$. Suppose $X$ satisfies:
	\begin{align*}
		\E(e^{\lambda X}) \le e^{\mu\lambda + \rho\lambda^2/2},
	\end{align*}
	for all $\lambda>0$ and some constant $\rho>0$. Then we have for all $s > 0$:
	\begin{align*}
		\E(|X|^s) &\le 2s!e^{(s + \lambda\mu)/2}\lambda^{-s}, \\
		\text{where } \lambda &= \frac{1}{2\rho}\left( \sqrt{\mu^2+ 4s\rho} -\mu \right).
	\end{align*}
	Further, if $\mu = 0$, we have for $t < 1/e\rho$:
	\begin{align*}
		\E\left( e^{tX^2} \right) &\le \frac{1}{1-e\rho t}
	\end{align*}
\end{lemma}
\begin{proof}
	By an application of Markov inequality and the given condition on $X$:
	\begin{align*}
		\P(X\ge t) &\le e^{-\lambda t}\E(e^{\lambda X}) \\
		&\le e^{-\lambda t + \mu \lambda + \rho\lambda^2/2},
	\end{align*}
	for all $\lambda > 0$. By a symmetric argument:
	\begin{align*}
		\P(X \le -t) &\le e^{\lambda t + \mu \lambda + \rho\lambda^2/2}
	\end{align*}
	By the standard integration formula we have:
	\begin{align*}
		\E(|X|^s) &= \int_0^\infty\! st^{s-1}\P(|X|\ge t)\,\mathrm{d}t \\
		&= \int_0^\infty \!st^{s-1}\P(X\ge t)\,\mathrm{d}t + \int_0^\infty\! st^{s-1}\P(X\le -t)\,\mathrm{d}t\\
		&\le 2e^{\mu\lambda+ \rho\lambda^2/2} \int_0^\infty\! st^{s-1}e^{-\lambda t}\,\mathrm{d}t \\
		&= 2 s!\, e^{\mu\lambda+\rho\lambda^2/2}\lambda^{-s}
	\end{align*}
	Optimizing over $\lambda$ yields the desired result.

	If $\mu = 0$, the optimization yields $\lambda = \sqrt{s/\rho}$. Using this,
	the Taylor expansion of $g(x) = e^{x^2}$ and monotone convergence we get:
	\begin{align*}
		\E\left( e^{tX^2} \right) &= \sum_{k=0}^\infty \frac{t^k}{k!} \E(X^{2k})\\
		&\le \sum_{k=0}^\infty (e\rho t)^k \frac{(2k)!}{k!(2k)^k} \\
		&\le \sum_{k=0}^\infty (e\rho t)^k \\
		&= \frac{1}{1-e\rho t}.
	\end{align*}
	Notice that here we remove the factor of $2$ in the inequality, since this
	is not required for even moments of $X$.
\end{proof}

The following lemma is standard, see for instance
\cite{alon2002concentration,Vershynin-CS}. 
\begin{lemma}
	Let $M \in \reals^{N\times N}$ be a symmetric matrix with entries $M_{ij}$
	(for $i\ge j$) which are centered subgaussian random variables of scale 
	factor $\rho$. Then, uniformly in $N$:
	\begin{align*}
		\P\left( \lVert M\rVert_2 \ge t \right) \le (5\lambda)^N e^{-N(\lambda - 1)},
	\end{align*}
	where $\lambda = t^2/16N\rho e$ and  $\norm{M}_2$ denotes the spectral norm (or largest 
	singular value) of $M$.
	\label{lem:2normbnd}
\end{lemma}
\begin{proof}
	Divide $M$ into its upper and lower triangular portions $M^u$ and $M^l$ so
	that $M = M^u+M_l$. We deal with each separately.  Let $m_i$ denote the 
	$i^{\text{th}}$ row of $M^l$. For a unit vector $x$, since $M_{ij}$ are
	all independent and subgaussian with scale $\rho$, it is easy to see that 
	$\<m_j, x\>$ are also subgaussian with the same scale. We now bound the 
	square exponential moment of $\norm{M^lx}$ as follows. For small enough $c\ge 0$:
	\begin{align}
		\E\left(e^{c\norm{M^lx}^2}\right) &= \E\left(\prod_{j=1}^N e^{c\<m_j, x\>^2} \right) \nonumber\\
		&= \prod_{j=1}^N\E\left( e^{c\<m_j, x\>^2} \right) \nonumber \\
		&\le \left( 1 - e\rho c \right)^N.\label{eq:Mlsqexp}
	\end{align}
	Using this, we get for any unit vector $x$:
	\begin{align*}
		\P(\norm{M^l x} \ge t ) &\le \left(\frac{t^2}{N\rho e}\right)^N e^{-N(t^2/N\rho e - 1)},
	\end{align*}
	where we used Markov inequality and \myeqref{eq:Mlsqexp} with an appropriate $c$.
	Let $\Upsilon$ be a \emph{maximal} $1/2$-net of the unit sphere. From a volume
	packing argument we have that $|\Upsilon| \le 5^N$. Then 
	from the fact that $g(x) = M^lx$ is $\norm{M^l}$-Lipschitz in $x$:
	\begin{align*}
		\P\left( \norm{M^l}_2 \ge t \right) &\le \P\left(\max_{x\in\Upsilon}\norm{M^lx} \ge t/2\right)\\
		&\le |\Upsilon| \P(\norm{M^lx} \ge t/2). 
	\end{align*}
	The same inequality holds for $M^u$. Now using the fact that $\norm{\cdot}_2$
	is a convex function and that $M^u$ and $M^l$ are independent we get:
	\begin{align*}
		\P\left(\norm{M}_2 \ge t \right) &\le \P\left( \norm{M^u}_2 \ge t/2 \right) + \P\left( \norm{M^l}_2 \ge t/2 \right) \\
		&\le 2 \left( 5^N \left(\frac{t^2}{16N\rho e}\right)^N e^{-N(t^2/16N\rho e - 1)} \right).
	\end{align*}
	Substituting for $\lambda$ yields the result.
\end{proof}

\section{Additional Proofs}

In this section we provide, for the sake of completeness, 
some additional proofs that are known results. We begin 
with Proposition \ref{propo:Spectral}.

\subsection{Proof of Proposition \ref{propo:Spectral}}\label{sec:SpectralProof}

We assume the set $\sC_N$ is generated as follows: let $X_i\in\{0, 1\}$ be the label
of the index $i\in [N]$. Then $X_i$ are i.i.d Bernoulli with parameter $\kappa/\sqrt{N}$
and the set $\sC_N = \{i : X_i =  1\}$. The model of choosing $\sC_N$ uniformly random of 
size $\kappa\sqrt{N}$ is similar to this model and asymptotically in $N$ there is no 
difference. Notice that since $e_{\sC_N} = u_{\sC_N}/N^{1/4}$, we have that 
$\|e_{\sC_N}\|^2$ concentrates sharply around $\kappa$ and we are interested in 
the regime $\kappa =\Theta(1)$.

We begin with the first part of the proposition where $\kappa = 1+\eps$.
Let $W_N = W/\sqrt{N}$, $Z_N = Z/\sqrt{N}$ and $e_{\sC_N} = u_{\sC_N}/N^{1/4}$. Since
this normalization does not make a difference to the eigenvectors of $W$ and
$Z$ we obtain from the eigenvalue equation $W_N v_1 = \lambda_1 v_1$ that:
\begin{align}\label{eq:EigenvalueEquation}
  e_{\sC_N}\<e_{\sC_N}, v_1\> + Z_Nv_1 &= \lambda_1v_1.
\end{align}
Multiplying by $v_1$ on either side:
\begin{align*}
  \<e_{\sC_N}, v_1\>^2 &= \lambda_1 - \<v_1, Z_Nv_1\>\\
  &\ge \lambda_1 - \|Z_N\|_2.
\end{align*}
Since $Z_N = Z/\sqrt{N}$ is a standard Wigner matrix with subgaussian
entries,  \cite{alon2002concentration}
yields that $\norm{Z}_2 \le 2 + \delta$ with
probability at least $C_1e^{-c_1N}$ for some constants $C_1(\delta),
c_1(\delta)>0$. Further, by Theorem 2.7 of \cite{knowles2011isotropic}
we have that $\lambda_1 \ge 2 + \min(\eps,\eps^2)$ with probability at least $1 - N^{-c_2\log\log N}$
for some constant $c_2$ and every $N$ sufficiently large.
It follows from this and the union bound that for $N$ large enough, we have:
\begin{align*}
  \< e_{\sC_N}, v_1\>^2 &\ge \min(\eps,\eps^2)/2,
\end{align*}
with probability at least $1 - N^{-c_4}$ for some constant $c_4>0$. The first claim then follows. 

For the second claim, we start with the same eigenvalue equation
(\ref{eq:EigenvalueEquation}).
Multiplying on either side by $\varphi_1$, the eigenvector corresponding
to the largest eigenvalue of $Z_N$ we obtain:
\begin{align*}
  \<e_{\sC_N}, v_1\>\<e_{\sC_N}, \varphi_1\> + \theta_1\<v_1,
  \varphi_1\> &= \lambda_1 \<v_1,\varphi_1\>\, , 
\end{align*}
where $\theta_1$ is the eigenvalue of $Z_N$ corresponding to $\varphi_1$. With this and
Cauchy-Schwartz we obtain:
\begin{align*}
  |\<e_{\sC_N}, v_1\>| &\le \frac{|\lambda_1 - \theta_1|}{|\<\varphi_1, e_{\sC_N}\>|}.
\end{align*}
Let $\phi = (\log N)^{\log\log N}$. Then, using Theorem 2.7 of \cite{knowles2011isotropic}, 
for any $\delta > 0$, there exists a constant $C_1$ such that $|\lambda_1 - \theta_1| \le N^{-1 + \delta}$
with probability at least $1 - N^{-c_3\log\log N}$. 

Since $\varphi_1$ is independent of $e_{\sC_N}$,
we observe that:
\begin{align*}
  \E_e\left(\sum_{i=1}^N \varphi_1^i e_{\sC_N}^i  \right) &= N^{-3/4}(1-\eps) \sum_{i=1}^{N}\varphi^i_1\\
  \E_{e}\left( \sum_{i=1}^N (\varphi_1^i e_{\sC_N}^i)^2 \right) &= \frac{1-\eps}{N},
\end{align*}
where $\varphi_1^i$ ($e_{\sC_N}^i$) denotes the $i^{\mathrm{th}}$ entry of $\varphi_1$ (resp. $e_{\sC_N}$)
and $\E_e(\cdot)$ is the expectation with respect to $e_{\sC_N}$ holding $Z_N$ (hence $\varphi_1$)
constant. Using Theorem 2.5 of \cite{knowles2011isotropic}, it follows that there exists constants
$c_4, c_5, c_6, c_7$ such that the following two happen with probability at least $1- N^{-c_4\log\log N}$. 
Firstly, the first expectation above is at most $(1-\eps)\phi^{c_5}N^{-7/4}$. Secondly:
\begin{align*}
  \left[\E_e\left(\sum_{i=1}^N (\varphi_1^ie_{\sC_N}^i)^2  \right)\right]^{-1/2} \max_i |e_{\sC_N}^i\varphi^1_i| &\le \frac{(1-\eps)\phi^{c_7}}{N^{1/4}}.
\end{align*}
Now, using the Berry-Esseen central limit theorem for $\<\varphi_1, e_{\sC_N}\>$ that: 
\begin{align*}
  \P\left( |\<\varphi_1, e_{\sC_N}\> |\le c(N)^{1/2-\delta}\right) \le \frac{1}{N^\delta},
\end{align*}
for an appropriate constant $c = c(\eps)$ and $\delta\in(0,1/4)$. Using this and the earlier bound for $|\lambda_1-\theta_1|$
we obtain that:
\begin{align*}
  |\<e_{\sC_N}, v_1\>| &\le cN^{-1/2+3\delta}
\end{align*}
with probability at least $1 - c'N^{-\delta}$, for some $c'$ and sufficiently large $N$.
The claim then follows using the union bound and the same argument for the first $\ell$
eigenvectors.

\subsection{Proof of Proposition \ref{propo:TreeError}}\label{sec:TreeErrorProof}
\def\cE{{\mathcal{E}}}
  For any fixed $t$, let $\cE_N^t$ denote the set of vertices in $G_N$ such that
  their $t$-neighborhoods are \emph{not} a tree, i.e. 
  \begin{align*}
    \cE^t_N &= \{i\in[N]:\Ball_{G_N}(i; t) \text{ is not a tree}\}.
  \end{align*}
  For notational simplicity, we will omit the subscript $G_N$ in the neighborhood
  of $i$. 
  The relative size $\eps^t_N = |\cE^t_N|/N$ vanishes asymptotically
  in $N$ since the sequence $\{G_N\}_{N\ge 1}$ is locally tree-like. We let $\sF_{BP}(W_{\Ball(i;t)})$
  denote the decision according to belief propagation at the $i^{\mathrm{th}}$ 
  vertex. 
  
  From Proposition 
  \ref{propo:BP_Exact_Tree}, Eqs.~(\ref{eq:BPeqnSparse1}), (\ref{eq:BPeqnSparse2}), (\ref{eq:TgammaDef}), 
  (\ref{eq:GammaDef}) and induction, we observe that for any
  $i\in[N]\bs\cE^t_N$:
  \begin{align*}
    \frac{\P(X_i = 1\rvert W_{\Ball(i;t)})}{\P(X_i = 0 \rvert
      W_{\Ball(i;t)})} &\ed \frac{\tgamma^t(X_i)}{\sqrt{\Delta}}.
  \end{align*}
  We also have that:
  \begin{align*}
    |\hsC_N\triangle\sC_N| &= \sum_{i=1}^N \ind(\sF_{BP}(W_{\Ball(i;t)}) \ne X_i).
  \end{align*}
  Using both of these, the fact that $\eps^t_N\to 0$ and the linearity 
  of expectation, we have the first claim:
  \begin{align*}
    \lim_{N\to\infty}\frac{\E|\hsC_N\triangle\sC_N|}{N} &= \frac{\tkappa}{\sqrt{\Delta}} \P\left( \gamma^t(1) < \sqrt{\Delta} \right)%
    + \left( 1- \frac{\tkappa}{\sqrt{\Delta}} \right) \P\left( \gamma^t(0) \ge \sqrt\Delta \right). 
  \end{align*}
  
  For any other decision rule $\sF(W_{\Ball(i;t)})$, we have that:
  \begin{align*}
    \frac{\E[ |\hsC_N\triangle\sC_N| ]}{N} &\ge (1-\eps^t_N) \P( \sF(W_{\tTree(t)}) \ne X_\root ) \\
    &\ge (1-\eps^t_N) \P( \sF_{BP}(W_{\tTree(t)}) \ne X_\root ),
  \end{align*}
  since BP computes the correct posterior marginal on the root 
  of the tree $\tTree(t)$ and maximizing the posterior marginal
  minimizes the misclassification error. The second claim follows by taking 
  the limits.

  \subsection{Equivalence of i.i.d and uniform set model}\label{sec:iidProof}
\def\si{ {\sf i}}

In Section \ref{sec:Complete} the hidden set $\sC_N$
was assumed to be uniformly random given its size. However, in Section \ref{sec:Sparse}
we considered a slightly different model to choose $\sC_N$, wherein $X_i$ are i.i.d Bernoulli
random variables with parameter $\tkappa/\sqrt{\Delta}$. This
leads to a set $\sC_N = \{i: X_i = 1\}$ that has a random size, 
sharply concentrated around $N\tkappa/\sqrt{\Delta}$. The uniform
set model can be obtained from the i.i.d model by simply conditioning
on the size $|\sC_N|$. In the limit of large $N$ it is well-known 
that these two models are ``equivalent''.
However, for completeness, we provide a proof that the results of
Proposition \ref{propo:TreeError} do not change 
when conditioned on the size $|\sC_N| = \sum_{i = 1}^{N}X_i$. 
\begin{align*}
\E\left[|\hsC_N\triangle\sC_N|\,\bigg\rvert\, |\sC_N| = \frac{N\tkappa}{\sqrt{\Delta}} \right]&=
\sum_{i=1}^N \P\bigg( \sF(W_{\Ball(i;t)}) \ne X_i \,\bigg\rvert\, \sum_{j=1}^N X_j = \frac{N\tkappa}{\sqrt{\Delta}} \bigg).
\end{align*}

Let $S$ be the event $\{\sum_{i=1}^N X_i=N\tkappa/\sqrt{\Delta}\}$. 
Notice that $\sF(W_{\Ball(i;t)})$ is a function of $\{X_j, j\in \Ball(i;t)\}$ which 
is a discrete vector of dimension $K_t \le (\Delta+1)^{t+1}$. A straightforward direct
calculation yields that $(X_j, j\in \Ball(i;t))\rvert S $ converges in distribution
to $(X_j, j\in \Ball(i;t))$ asymptotically in $N$.
This implies $W_{\Ball(i;t)} \rvert S$ converges in distribution to $W_{\Ball(i;t)}$.
Further, using the locally tree-like property of $G_N$ one obtains:
\begin{align*}
  \lim_{N\to \infty} \frac{1}{N}\E\left[|\hsC_N\triangle\sC_N|\,\bigg\rvert\, |\sC_N| = \frac{N\tkappa}{\sqrt{\Delta}} \right]&= %
  \P\left( \sF(W_{\tTree(t)})\ne X_\root \right),
\end{align*}
as required.

\bibliographystyle{amsalpha}

\addcontentsline{toc}{section}{References}

\newcommand{\etalchar}[1]{$^{#1}$}
\providecommand{\bysame}{\leavevmode\hbox to3em{\hrulefill}\thinspace}
\providecommand{\MR}{\relax\ifhmode\unskip\space\fi MR }
\providecommand{\MRhref}[2]{%
  \href{http://www.ams.org/mathscinet-getitem?mr=#1}{#2}
}
\providecommand{\href}[2]{#2}

\end{document}